\documentclass[a4paper,12pt]{article}

\usepackage[T1]{fontenc}
\usepackage{amsmath}
\usepackage{amssymb}
\usepackage{amsfonts}
\usepackage{amsthm}
\usepackage{indentfirst}
\usepackage{psfrag}
\usepackage{amscd,stmaryrd,latexsym}

\newtheorem{thm}{Theorem}[section]
\newtheorem{lem}{Lemma}[section]

\newtheorem{Prop}{Proposition}

\theoremstyle{remark}

\theoremstyle{definition}

\theoremstyle{remark}
\newtheorem{oss}{Remark}[section]

\def\eps{\mathop{\varepsilon}}

\newcommand{\be}{\begin{equation}}
\newcommand{\ee}{\end{equation}}
\newcommand{\R}{\mathbb{R}}
\newcommand{\N}{\mathbb{N}}
\newcommand{\E}{\mathbb{E}}
\newcommand{\T}{\mathbb{T}}
\newcommand{\C}{\mathbb{C}}
\newcommand{\CP}{\mathbb{C}\mathbb{P}}

\newcommand\res{\mathop{\hbox{\vrule height 7pt width .5pt depth 0pt
\vrule height .5pt width 6pt depth 0pt}}\nolimits}

\def\Om{\Omega}
\def\om{\omega}
\def\p{\partial}

\def\omtp{\tilde{\omega}_+}

\begin{document}

\title{\textbf{Almost complex structures and calibrated integral cycles in contact $5$-manifolds}}
\author{\textit{Costante Bellettini}\footnote{Address: D-MATH, ETH Zentrum, R\"amistrasse 101, 8092 Z\"urich, Switzerland.}} 
\date{}
\maketitle

\textbf{Abstract}: \textit{In a contact manifold $(\mathcal{M}^5, \alpha)$, we consider almost complex structures $J$ which satisfy, for any vector $v$ in the horizontal distribution, $d\alpha(v,Jv)=0$. We prove that integral cycles whose approximate tangent planes have the property of being $J$-invariant are in fact smooth Legendrian curves except possibly at isolated points and we investigate how such structures $J$ are related to calibrations.}

\section{Introduction}

In the last fifteen years, there has been a growing interest regarding links between the theory of vector bundles and the geometry of submanifolds. Striking examples of intimate correlations were found in many geometrical and physical problems. Donaldson and Thomas exhibited in \cite{DT} relations between some invariants in complex geometry and spaces of solutions to Yang-Mills equations. Tian showed in \cite{T} that some particular sequences of Yang-Mills fields present a loss of compactness along calibrated recti\-fiable currents. Taubes (see \cite{Ta}) proved that Seiberg-Witten invariants in a symplectic 4-manifold coincide with Gromov invariants. Mirror symmetry (see in particular \cite{SYZ}) described, in the framework of a String Theory model, a phenomenon regarding Special Lagrangian cycles (see \cite{J} for an overview).

A common feature in these situations is the important role played by the so-called \textit{calibrations}. This notion is strongly related to the theory of minimal submanifolds. For a history of calibrations the reader may consult \cite{M}. In the fundational essay \cite{HL} the authors exhibited and studied several rich ``Calibrated geometries''.

In \cite{BR}, together with T. Rivi\`ere, we analyzed the regularity of Special Legendrian integral cycles in $S^5$. In this work that result is generalized to contact 5-manifolds with certain almost complex structures.

\medskip

\textbf{Setting and main result}.
Let $\mathcal{M}=\mathcal{M}^5$ be a five-dimensional manifold endowed with a contact structure\footnote{For a broader exposition on contact geometry, the reader may consult \cite{B} or \cite{MS}.} defined by a one-form $\alpha$ which satisfies everywhere 
\be
\label{eq:defcontact}
\alpha \wedge (d \alpha)^2 \neq 0.
\ee

Remark that the existence of a contact structure implies the orientability of $\mathcal{M}$. We will assume $\mathcal{M}$ oriented by the top-dimensional form $\alpha \wedge (d \alpha)^2$.

Condition (\ref{eq:defcontact}) means that the horizontal distribution $H$ of $4$-dimensional hyperplanes $\{H_p\}_{p\in M}$ defined by
\be
H_p:= Ker \; \alpha_p 
\ee
is ``as far as possible'' from being integrable. The integral submanifolds of maximal dimension for the contact structure are of dimension two and are called \textit{Legendrians}.

\medskip

Given a contact structure, there is a unique vector field, called the Reeb vector field $R_\alpha$ (or vertical vector field), which satisfies $\alpha(R_\alpha)=1$ and $\iota_{R_\alpha} d\alpha =0$. 

\medskip
An almost-complex structure on the horizontal distribution is and endomorphism $J$ of the horizontal sub-bundle which satisfies $J^2=-Id$. Given a horizontal, non-degenerate two-form $\beta$, i.e. a two-form such that $\iota_{R_\alpha} \beta =0$ and $\beta \wedge \beta \neq 0$, we say that an almost-complex structure $J$ is compatible with $\beta$ if the following conditions are satisfied:

\be
\label{eq:compatibility}
\beta(v,w)=\beta(Jv, Jw), \;\; \beta(v,Jv)>0 \;\; \text{ for any } v,w \in H.
\ee
In this situation, we can define an associated Riemannian metric on the horizontal sub-bundle by $$g_{J,\beta}(v,w)=\beta(v,Jw).$$ 

\medskip

We can extend an almost-complex structure $J$ defined on the horizontal distribution, to an endomorphism of the tangent bundle $T \mathcal{M}$ by setting
\begin{equation}
\label{eq:Jextended}
J(R_\alpha)=0.
\end{equation}
Then it holds $J^2= -Id + R_\alpha \otimes \alpha$.

With this in mind, extend the metric to a Riemannian metric on the tangent bundle by

\begin{equation}
g:=g_{J,\beta} + \alpha \otimes \alpha.
\end{equation}

This extensions will often be implicitly assumed. Remark that $R_\alpha$ is orthogonal to the hyperplanes $H$ for the metric $g$:
\be
g(R_\alpha, X)=\beta(R_\alpha, J X) + \alpha(R_\alpha)  \alpha(X) = 0 \;\text{ for } X\in H= Ker \;\alpha.
\ee

\underline{Example}: We describe the \textit{standard contact structure} on $\R^5$. Using coordinates $(x_1, y_1, x_2, y_2, t)$ the standard contact form is $\zeta=dt-(y_1 dx^1 +y_2 dx^2)$. The expression for $d \zeta$ is $dx^1 d y^1 +  dx^2 d y^2$ and the horizontal distribution is given by 

\be
\label{eq:standardker}
Ker \;\zeta = Span \{\p_{x_1}+y_1 \p_t, \p_{x_2}+y_2 \p_t, \p_{y_1}, \p_{y_2}\}.
\ee

The standard almost complex structure $I$ compatible with $d \zeta$ is the endomorphism
\be
\label{eq:I}
\left \{ \begin{array}{c} 
I(\p_{x_i}+y_i \p_t)= \p_{y_i}\\
I(\p_{y_i})= -(\p_{x_i}+y_i \p_t) \end{array} \right . i\in\{1,2\}.
\ee

$I$ and $d\zeta$ induce the metric $g_\zeta:=d \zeta(\cdot, I \cdot)$ for which the hyperplanes $Ker \;\zeta$ are orthogonal to the $t$-coordinate lines, which are the integral curves of the Reeb vector field. The metric $g_\zeta$ projects down to the standard euclidean metric on $\R^4$, so the projection

\be
\begin{array}{ccc}
\pi: \R^5 & \rightarrow & \R^4 \\
(x_1, y_1, x_2, y_2, t) & \rightarrow & (x_1, y_1, x_2, y_2)\end{array}
\ee

is an isometry from $(\R^5, g_\zeta)$ to  $(\R^4, g_{\text{eucl}})$.

\medskip
 
The following will be useful in the sequel:
\begin{oss}
\label{oss:parall}
Observing (\ref{eq:standardker}), we can see that, for any $q \in \R^4$, all the hyperplanes $H_{\pi^{-1}(q)}$ are parallel in the standard euclidean space $\R^5$. Thus the lift of a vector in $\R^4=\{t=0\}$ with base-point $q$ to an horizontal vector based at any point of the fiber $\pi^{-1}(q)$ has always the same coordinate expression along this fiber.
\end{oss}

We will be interested in two-dimensional Legendrians which are invariant for suitable almost complex structures defined on the horizontal distribution. What we require for an almost-complex structure $J$ on the horizontal sub-bundle is the following \textbf{Lagrangian condition}

\be
\label{eq:Jlagr}
d \alpha (J v, v)= 0 \; \text{ for any } v \in H.
\ee
This requirement amounts to asking that any $J$-invariant 2-plane must be Lagrangian for the symplectic form $d \alpha$.
It is also equivalent to the following \textbf{anti-compatibility condition}
\be
\label{eq:anticomp}
d \alpha (v, w)= - d \alpha (J v, Jw) \; \text{ for any } v,w \in H.
\ee

It is immediate that (\ref{eq:anticomp}) implies (\ref{eq:Jlagr}). On the other hand, using (\ref{eq:Jlagr}):
$$0= d \alpha (J (v+w), v+w)= d \alpha (J v, v) + d \alpha (J w, w) + d \alpha (J v, w) + d \alpha (J w, v) =$$ $$= d \alpha (J v, w) + d \alpha (J w, v)$$
so
$$d \alpha (J v, w) = d \alpha (v,J w) \;\; \text{for any horizontal vectors } v \text{ and } w.$$
Writing this with $J v$ instead of $v$, and being $J$ an endomorphism of $H$, we obtain (\ref{eq:anticomp}).

\medskip

By an \textbf{integral cycle} $S$ we mean an integer multiplicity rectifiable current without boundary. These are the generalized submanifolds of geometric measure theory. They have the fundamental property of possessing at almost every point $x$ an \textit{oriented approximate tangent plane} $T_x S$. For the topic, we refer the reader to \cite{F} or \cite{G}. We will deal with integral cycles of dimension $2$.

\medskip

For the sequel, we recall the notion of \textbf{calibration}, confining ourselves to $2$-forms. For a broader exposition, and for the connections to mass-minimizing currents, the reader is referred to \cite{HL}, \cite{M} and \cite{J}.

Given a $2$-form $\phi$ on a Riemannian manifold $(M,g)$, the comass of $\phi$ is defined to be
\[||\phi||_*:=  \sup \{\langle \phi_x, \xi_x \rangle: x \in M, \xi_x \text{ is a unit simple $2$-vector at } x\}.\]
A form $\phi$ of comass one is called a \textit{calibration} if it is closed ($d \phi = 0$); when it is non-closed it is referred to as a \textit{semi-calibration}. 

Let $\phi$ be a \textit{calibration} or a \textit{semi-calibration}; among the oriented $2$-dimen\-sional planes that constitute the Grassmannians $G(x, T_xM)$, we pick those that, represented as unit simple $2$-vectors, realize the equality $\langle \phi_x, \xi_x \rangle = 1$. Define the set $\mathcal{G}(\phi)$ of  \textit{$2$-planes calibrated by $\phi$} as
\[\mathcal{G}(\phi) = \cup_{x \in M} \{\xi_x \in G(x, T_xM): \langle \phi_x, \xi_x \rangle = 1 \}.\]
Given a (semi)-calibration $\phi$, an integral cycle $S$ of dimension $2$ is said to be \textit{(semi)-calibrated by $\phi$} if 
\[\text{for } \mathcal{H}^2 \text{-almost every } x, \;T_x S \in \mathcal{G}(\phi).\]
If $d \phi=0$, a calibrated cycle is automatically homologically mass minimizing. However we will be generally concerned with semi-calibrations.

\medskip

The main result in this work is the following
\begin{thm}
\label{thm:main}
Let $\mathcal{M}$ be a five-dimensional manifold endowed with a contact form $\alpha$ and let $J$ be an almost-complex structure defined on the horizontal distribution $H= Ker \; \alpha$ such that $d \alpha (J v, v)= 0 \; \text{ for any } v \in H$. 

Let $C$ be an integer multiplicity rectifiable cycle of dimension $2$ in $\mathcal{M}$ such that $\mathcal{H}^2$-a.e. the approximate tangent plane $T_x C$ is $J$-invariant\footnote{Representing a $2$-plane as a simple $2$-vector $v \wedge w$, the condition of $J$-invariance means $v \wedge w= Jv \wedge Jw$. With (\ref{eq:Jextended}) in mind, we see that a $J$-invariant $2$-plane must be tangent to the horizontal distribution.}. 

Then $C$ is, except possibly at isolated points, the current of integration along a smooth two-dimensional Legendrian curve.
\end{thm}

\medskip

In \cite{BR}, together with T. Rivi\`ere, we proved the corresponding regularity property for Special Legendrian Integral cycles\footnote{These cycles are briefly described in section \ref{ex}, where the reader may also find other examples where theorem \ref{thm:main} applies.} in $S^5$. From Proposition 2 in \cite{BR}, it follows that theorem \ref{thm:main} applies in particular to Special Legendrians in $S^5$ and therefore generalizes that result (also compare Proposition \ref{Prop:calibr} of the present paper, where we describe a direct application of this theorem to semi-calibrations).

\medskip

In this work we looked for a natural general setting in which an analysis analogous to the one in \cite{BR} could be performed. This lead to the assumptions taken above, in particular to conditions (\ref{eq:Jlagr}) and (\ref{eq:anticomp}).

The key ingredient that we need for the proof of theorem \ref{thm:main} is the construction of families of $3$-dimensional surfaces\footnote{This existence result is where (\ref{eq:Jlagr}) and (\ref{eq:anticomp}) play a determinant role.} which locally foliate the 5-dimensional ambient manifold and that have the property of \textbf{intersecting positively} the Legendrian, $J$-invariant cycles.
In \cite{BR}, due to the fact that we were dealing with an explicit semi-calibration in a very symmetric situation, the $3$-dimensional surfaces could be explicitly exhibited (section 2 of \cite{BR}). Here we will achieve this by solving, via fixed point theorem, a perturbation of Laplace's equation. After having achieved this, the proof can be completed by following that in \cite{BR} verbatim.

The same idea was present in \cite{RT1}, where, in an almost complex $4$-manifold, the authors produced $J$-holomorphic foliations by solving a perturbed Cauchy-Riemann equation. In the present work, the equation turns out to be of second order and, in order to prove the existence of a solution, we need to work in adapted coordinates (see proposition \ref{Prop:uno} and the discussion which precedes it).

We remark that the proof of the regularity result, an overview of which is presented at the end of section \ref{proof}, basically follows the structure of \cite{Ta} and \cite{RT1}, with the due changes since there is a fifth dimension to deal with, which introduces new difficulties (compare the introduction of \cite{BR}, pages 6-7).

\medskip

The results needed for the proof of theorem \ref{thm:main} are in section \ref{proof} and the reader may go straight to that. In section \ref{formsandJ} we show the existence of $J$-structures satisfying (\ref{eq:Jlagr}) or (\ref{eq:anticomp}) and discuss how they are related to 2-forms, in particular to semi-calibrations. In the last section we discuss examples and possible applications of theorem \ref{thm:main}.

\medskip

\textbf{Aknowledgments}: I am very grateful to Tristan Rivi\`ere for having encouraged me to work on this topic and for always being avaliable for discussion with helpful comments and suggentions.

\section{Almost complex structures and two-forms.}
\label{formsandJ}

\subsection{Self-dual and anti self-dual forms.}

On a contact 5-manifold $(\mathcal{M}, \alpha)$, take an almost-complex structure $I$ compatible with the symplectic form $d \alpha$, and let $g$ be the metric defined by $g(v,w):= d \alpha(v,I w) + \alpha \otimes \alpha$.

The metric $g$ induces a metric on the horizontal sub-bundle $H$, which also inherits an orientation from $\mathcal{M}$. 

Any horizontal two-form can be split in its self-dual and anti self-dual parts as follows. 

Let $\ast$ be the Hodge-star operator acting on the cotangent bundle $T^* \mathcal{M}$. Define the operator
\be
\star : \Lambda^2(T \mathcal{M}) \rightarrow \Lambda^2(T \mathcal{M}), \;\;\;
\star(\beta):= \ast (\alpha \wedge \beta),
\ee
and remark that $\star$ naturally restricts to an automorphism of the space of horizontal forms $\Lambda^2(H)$:
\be
\star : \Lambda^2(H) \rightarrow \Lambda^2(H), \;\;\;
\star(\beta):= \star (\alpha \wedge \beta).
\ee
This operator satisfies $\star^2=id$.

This yields the orthogonal eigenspace decomposition
\be
\Lambda^2(H)= \Lambda_+^2(H) \oplus \Lambda_-^2(H),
\ee
where $\Lambda_{\pm}^2(H)$ is the eigenspace relative to the eigenvalue $\pm 1$ of $\star$. These eigenspaces are referred to as the space of \textit{self-dual} and the space of \textit{anti self-dual} two-forms\footnote{This is basically the definition of anti self-duality used in \cite{T}.}.

\medskip 

In other words, we can restrict to the horizontal sub-bundle with the inherited metric and orientation and define the Hodge-star operator on horizontal forms by using the same definition as the general one, but confining ourselves to the horizontal forms. We get just the $\star$ defined above.

\subsection{From a two-form to $J$.}

In a contact 5-manifold $(\mathcal{M}, \alpha)$, given a horizontal two-form $\omega$ (with some conditions), is there an almost complex structure compatible with $\omega$ and satisfying $d \alpha (J v, v)= 0 \; \text{ for any } v \in H$ ?

\medskip

In this section, by answering positively the above question, we will also estabilish the \underline{existence} of such \textit{anti-compatible} almost complex structures.

\medskip

Assume that, on a contact 5-manifold $(\mathcal{M}, \alpha)$, a two-form $\omega$ is given, which satisfies 
\be
\label{eq:orthogforms}
\omega \wedge d \alpha =0 
\ee
and 
\be
\label{eq:nonzerocond}
\omega \wedge \omega \neq 0.
\ee
Conditions (\ref{eq:orthogforms}) and (\ref{eq:nonzerocond}) automatically give that $\om$ is horizontal\footnote{This can be checked in coordinates pointwise. Alternatively one can adapt the proof of \cite{CS}, Proposition 2.}, $\iota_{R_\alpha}\om=0$.
Without loss of generality, we may assume that 
\be
\label{eq:positiveorient}
\omega \wedge \omega = f (d\alpha)^ 2 \text{ for a strictly positive}\footnote{Indeed, the non-zero condition in (\ref{eq:nonzerocond}) implies that (\ref{eq:positiveorient}) holds with $f$ either everywhere positive or everywhere negative. The case $f<0$ can be treated after a change of orientation on $\mathcal{M}$ just in the same way.} \text{ function $f$.}
\ee

Take an almost-complex structure $I$ compatible with the symplectic form $d \alpha$, and let $g=g_{d\alpha, I}$ be the metric defined by $g(v,w):= d \alpha(v,I w) + \alpha \otimes \alpha$.

Decompose $\omega=\omega_+ + \omega_-$, where $\omega_+$ is the self-dual part and $\omega_-$ is the anti self-dual part. By definition $d\alpha$ is self-dual for $g$, so we have $$\omega_- \wedge d\alpha=\langle \omega_- , d\alpha \rangle d vol_g =0 $$
since $\Lambda_+^2$ and $\Lambda_-^2$ are orthogonal subspaces. Therefore (\ref{eq:orthogforms}) can be restated as

\be
\omega_+ \wedge d \alpha =0.
\ee

Consider now the form\footnote{The notation $\|\;\|$ denotes here the standard norm for differential forms coming from the metric on the manifold. It should not be confused with the comass, which is denoted by $\|\;\|_*$. They are in general different: for example, in $\R^4$ with the euclidean metric and standard coordinates $(x_1, x_2, x_3, x_4)$, the $2$-form $\beta=dx^1 \wedge dx^2 + dx^3 \wedge dx^4$ has norm $\| \beta\|=\sqrt{2}$ and comass $\| \beta\|_*=1$.} $$\omtp := \frac{\sqrt{2}}{\|\omega_+\|}\omega_+ .$$ 
It is self-dual and of norm $\sqrt{2}$, so there is an unique almost complex structure on the horizontal bundle which is compatible with $g$ and $\omtp$. It is defined by
$$J:=g^{-1}(\omtp).$$

We want to show that $d \alpha (J v, v)= 0 \; \text{ for any } v \in H$.

To this aim, it is enough to work pointwise in coordinates. We can choose an orthonormal basis for $H$ (at the chosen point) of the form $\{e_1=X, e_2=IX, e_3=Y, e_4=IY\}$ and denote by 

\be
\label{eq:basisI}
\{e^1, e^2, e^3, e^4\}
\ee

 the dual basis of orthonormal one-forms. Then $d \alpha$ has the form $e^{12} + e^{34}$, where we use $e^{ij}$ as a short notation for $e^i \wedge e^j$. The forms $e^{12} + e^{34}$, $e^{13} + e^{42}$ and $e^{14} + e^{23}$ are an orthonormal basis for $\Lambda_+^2$. The fact that $\omega_+$ is orthogonal to $d \alpha$ implies that 
\be
\omega_+ = a (e^{13} + e^{42}) + b (e^{14} + e^{23}).
\ee
and $\|\omega_+\|^2= 2(a^2 + b^2)$, therefore $\omtp= \cos \theta (e^{13} + e^{42}) + \sin \theta(e^{14} + e^{23}) $ for some $\theta$ depending on the chosen point, $\cos \theta= \frac{a}{\sqrt{a^2+b^2}}$, $\sin \theta= \frac{b}{\sqrt{a^2+b^2}}$. Then the explicit expression for $J$ is
\be
\label{eq:Jass}
\begin{array}{c}
J(e_1)=\cos \theta e_3 + \sin \theta e_4 \\
J(e_2)=-\cos \theta e_4 + \sin \theta e_3 \\
J(e_3)=-\cos \theta e_1 - \sin \theta e_2 \\
J(e_4)=\cos \theta e_2 - \sin \theta e_1  \end{array}
\ee
and an easy computation shows that $d \alpha(v, J(v))=0$ for any $v \in H$.

\medskip

The almost complex structure $J$ is also compatible with $\omega_+$, since this form is just a scalar multiple of $\omtp$, and the metric associated to $(\omtp, J)$ is $\frac{\|\omega_+\|}{\sqrt{2}} g$ when restricted to the horizontal bundle.

\medskip
\medskip

Let us now look at $\omega_-$. It is interesting to observe that 
\be
\label{eq:minusinvariance}
\omega_-(v,w) = \omega_-(J v,J w).
\ee
This can be once again checked pointwise in coordinates, as above. An orthonormal basis for $\Lambda_+^2$ is given by the forms $e^{12} - e^{34}$, $e^{13} - e^{42}$ and $e^{14} - e^{23}$, therefore $\omega_-$ is a linear combination of these forms, each of which can be checked to satisfy the invariance expressed in (\ref{eq:minusinvariance}) with respect to $J$.

On the other hand, these anti self-dual forms do not give a positive real number when applied to $(v, Jv)$ for an arbitrary $v \in H$. However, due to (\ref{eq:positiveorient}) we can show that $\om(v, Jv) = \om _+(v, Jv) + \om _-(v, Jv) >0$ for any $v$.

Indeed, write again
\be
\label{eq:selfdualpart}
\om_+ = a (e^{13} + e^{42}) + b (e^{14} + e^{23}) ,
\ee

\be
\label{eq:antiselfdualpart}
\om_- = A (e^{12} - e^{34}) + B (e^{13} - e^{42}) + C (e^{14} - e^{23}) .
\ee

So we can compute
\be
\label{eq:computewedge}
\om_+ \wedge \om_+= (a^2 + b^2) (d\alpha)^2 \;\; \text{ and } \;\; \om_- \wedge \om_-= -(A^2+B^2+ C^2) (d\alpha)^2.
\ee

Condition (\ref{eq:positiveorient}), recalling that $\om_+ \wedge \om_-=0$, then reads
\be
f (d\alpha)^2 = \om_+ \wedge \om_+ + \om_- \wedge \om_- = \left((a^2 + b^2) - (A^2+B^2+ C^2)\right)(d\alpha)^2 
\ee
with a positive $f$, so $(a^2 + b^2) > (A^2+B^2+ C^2)$.
\medskip
Observe that $$\om _-(e_i, J(e_i))= \pm B \cos \theta \pm C \sin \theta,$$ with $\cos \theta= \frac{a}{\sqrt{a^2+b^2}}$, $\sin \theta= \frac{b}{\sqrt{a^2+b^2}}$. We can bound $\pm B \cos \theta \pm C \sin \theta \leq \sqrt{B^2+C^2}$, so

 $$\om(e_i, J(e_i))=\om _+(e_i, J(e_i)) + \om _-(e_i, J(e_i))= \sqrt{a^2+b^2} + \om _-(e_i, J(e_i))$$ $$\geq  \sqrt{a^2+b^2} - \sqrt{A^2+B^2+C^2} >0.$$

This means that the almost complex structure $J$ is compatible with $\om$ in the sense of (\ref{eq:compatibility}) and they induce a metric $\tilde{g}(v,w):=\om(v,Jw)$ for which $J$ is orthogonal and $\om$ is self dual and of norm $\sqrt{2}$. This gives a positive answer to the question raised in the beginning of this section.

\medskip

Moreover we get

\begin{Prop}
Given a contact 5-manifold $(\mathcal{M}, \alpha)$, there exist almost complex structures $J$ such that $d \alpha(v,Jv) =0$ for all horizontal vectors $v$.
\end{Prop}

Indeed, we can get a two-form $\omega$ satisfying  (\ref{eq:orthogforms}) and (\ref{eq:positiveorient}). This can be done locally\footnote{For example, work on an open ball where we can apply Darboux's theorem (see \cite{B} or \cite{MS}), which allows us to work with the standard contact structure of $\R^5$ described in the introduction (compare the explanation at the beginning of section \ref{posfol}).} and then we can get a global form by using a partition of unity on $\mathcal{M}$. The previous discussion in this subsection then shows how to construct the requested almost complex structure from $\om$, thereby proving that anti-invariant almost complex structures exist.

Also remark that the almost complex structure $J$ anti-compatible with $d \alpha$ that we constructed is orthogonal for the metric $g$ associated to $d \alpha$ and $I$. Indeed, after having built the two-form $\omega$ satisfying  (\ref{eq:orthogforms}) and (\ref{eq:positiveorient}), we defined $J$ from its self-dual part (suitably rescaled) and from the metric $g$.

By changing the almost complex structure $I$ compatible with $d \alpha$, we can get different anti-compatible structures.

\medskip

We conclude with the following proposition, which gives a condition to ensure the applicability of theorem \ref{thm:main} to a semi-calibration.

\begin{Prop}
\label{Prop:calibr}
Let $(\mathcal{M}, \alpha)$ be a contact 5-manifold, with a metric $g$ defined\footnote{This is equivalent to asking that $d \alpha $ is self-dual and of norm $\sqrt{2}$ for $g$} by $g=d \alpha(\cdot, I \cdot)$ for an almost complex structure $I$ compatible with $d \alpha$.

Let $\om$ be a two-form of comass $1$, $\|\om\|_*=1$, such that:
$$\om \wedge d \alpha=0, \;\; \om \wedge \om= (d \alpha)^2. $$
Then $\om$ is self-dual with respect to $g$ and there is an almost complex structure $J$ anti-compatible with $d\alpha$ such that the (semi)-calibrated two-planes are exactly the $J$-invariant ones. 
\end{Prop}

Therefore theorem \ref{thm:main} applies to such an $\om$, yielding the regularity of $\om$-(semi)calibrated cycles. The Special Legendrian semi-calibration treated in \cite{BR} fulfils the requirements of Proposition \ref{Prop:calibr}.

\begin{proof}[\bf{proof of proposition \ref{Prop:calibr}}]
Decompose $\om= \om_+ + \om_-$ as in (\ref{eq:selfdualpart}) and (\ref{eq:antiselfdualpart}). Evaluating $\om$ on the unit simple 2-vector
$$\frac{1}{\sqrt{a^2+b^2+A^2}}e_1 \wedge (a e_3 + b e_4 + A e_2)$$
we get $\sqrt{a^2+b^2+A^2}$. The condition $\|\om\|_*=1$ implies $a^2+b^2+A^2 \leq 1$.

On the other hand, expliciting $\om \wedge \om= (d \alpha)^2$ as in (\ref{eq:computewedge}), we obtain
$$a^2+b^2-A^2-B^2-C^2=1. $$

Hence $-B^2-C^2 \geq 2 A^2$, which trivially yields $A=B=C=0$, so $\om$ is a self-dual form.

A self-dual form of comass $1$ expressed as in (\ref{eq:selfdualpart}) must be of the form
$$\cos \theta (e^{13} + e^{42}) + \sin \theta(e^{14} + e^{23})$$
and the corresponding almost-complex structure $J=g^{-1}(\om)$ has the expression (\ref{eq:Jass}) and is anti-compatible with $d \alpha$. 

Take two orthonormal vectors $v,w$. Since $\om(v,w)=\langle v, -Jw \rangle_g$, we have that 
$$\om(v,w)=1 \Leftrightarrow w=Jv ,$$
from which we can see that the semi-calibrated 2-planes are exactly those that are $J$-invariant for this $J$.

\end{proof}

\subsection{From $J$ to a two-form.}
\label{j2f}

In this subsection, we want to answer the following question, in some sense the natural reverse to the one raised in the previous subsection.

\medskip

Assume that, on a contact 5-manifold $(\mathcal{M}, \alpha)$, an almost-complex structure $J$ on the horizontal distribution $H$ is given, which satisfies $d \alpha (J v, v)= 0 \; \text{ for any } v \in H$. Is there a two-form which is compatible with $J$ in the sense of (\ref{eq:compatibility})?

\medskip

Let $I$ be an almost-complex structure compatible with the symplectic form $d \alpha$, and let $g$ be the metric on $\mathcal{M}$ defined by $g(v,w):= d \alpha(v,I w) + \alpha \otimes \alpha$. 

Define a two-form $\Omega$ by
\[\Omega(X,Y):=d \alpha(JX,\frac{1}{2}(JI-IJ)Y).\]
We can see that $\Omega$ is compatible with $J$ as follows

$$\begin{array}{c} \bullet \;\; \Omega(JX,JY) = \frac{1}{2} d \alpha(JX,JIJ Y +IY)= \\
\\
- d \alpha(X,\frac{1}{2} IJ Y) + d \alpha(X,\frac{1}{2}JIY)=\Omega(X,Y),\end{array}$$

$$\begin{array}{c} 
\bullet \;\; \Omega(X,JX) = \frac{1}{2} d \alpha(X,JIJ X +IX)= \\
\\
=\frac{1}{2} d \alpha(JX,IJX)+ \frac{1}{2} d \alpha(X,IX) >0, \end{array}$$

where we used the anti-compatibility property (\ref{eq:anticomp}).
It also follows that the metric $\tilde{g}(X,Y):=\Omega(X,JY) + \alpha \otimes \alpha$ is related to $g$ by $\tilde{g}(X,Y)=\frac{1}{2}(g(X,Y)+g(JX,JY))$ when restricted to the horizontal sub-bundle. 

\medskip

In local coordinates, also just pointwise for a basis of the form $\{e_1=X, e_2=IX, e_3=Y, e_4=IY\}$ as in (\ref{eq:basisI}), it can be checked by a direct computation that $\Om$ also satisfies $\Om \wedge d\alpha=0$.

\medskip

The two-form $\Omega$ is a semi-calibration on the manifold $\mathcal{M}$ endowed with the metric $\tilde{g}$. Indeed, since $J$ preserves the $\tilde{g}$-norm, for any two vectors $v,w$ at $p$ which are orthonormal with respect to $\tilde{g}$

$$\Omega(v,w)= \langle v,-J w\rangle_{\tilde{g}} \leq |v|_{\tilde{g}} |Jw|_{\tilde{g}}=|v|_{\tilde{g}} |w|_{\tilde{g}}=1.$$

and equality is realized if and only if $Jv=w$. This means that a 2-plane is $\Omega$-calibrated if and only if it is $J$-invariant, so theorem \ref{thm:main} applies to $\Omega$-semicalibrated cycles\footnote{We remark here that, being semi-calibrated, such a cycle will satisfy an almost-monotonicity formula at every point, as explained in \cite{PR}.}. 

\medskip

If $J$ is an orthogonal transformation with respect to $g$, the anti-compatibi\-lity with $d\alpha$ implies, for all horizontal vectors $X, Y$

$$\left\{\begin{array}{c} g(X,Y) = g(JX,JY)  \\
=d \alpha(JX, IJY)=d \alpha(X, JIJY) \\
=d \alpha(IX, (IJ)^2 Y)= -g(X,(IJ)^2 Y) \end{array} \right. \Rightarrow g(X,(Id+(IJ)^2) Y)=0 $$
so $(IJ)^2=-Id$ when restricted to $H$, therefore $IJ=-JI$. 

Hence $\Omega(X,Y):=d \alpha(X,JIY)$. In this case $\Omega$ is a self-dual form of norm $\sqrt{2}$ and comass $\|\Om\|_*=1$ with respect to $g$ \footnote{Observe that, in this case, we have that pointwise $\{Id, I, J, IJ\}$ form a quaternionic structure.}.

\section{Proof of theorem \ref{thm:main}}
\label{proof}

\subsection{Positive foliations}
\label{posfol}

The regularity property in Theorem \ref{thm:main} is local. It is therefore enough to prove the statement for an arbitrarily small neighbourhood $B^5(p) \subset \mathcal{M}$ of any chosen $p \in \mathcal{M}$.

\medskip

From Darboux's theorem, we know that there is a diffeomorphism\footnote{In the usual terminology, for example see \cite{MS} or \cite{B}, it is called a \textit{contactomorphism} or \textit{contact transformation}.} $\Phi$ from a ball centered at the origin of the standard contact manifold $(\R^5, dt-y_1 dx^1 -y_2 dx^2)$ to such a neighbourhood $B^5(p)$, with $\Phi^* (\alpha)= dt-(y_1 dx^1 +y_2 dx^2)$. The structure $J$ on $\mathcal{M}$ can be pulled back to an almost complex structure on $\R^5$ via $\Phi$:

$$(\Phi^* J)(X) :=  (\Phi^{-1})_* [J (\Phi_*X)] \;\;\text{ for } X\in \R^5. $$

Condition (\ref{eq:Jlagr}) yields

$$d (\Phi ^* \alpha )((\Phi^*J) X, X)= (\Phi^* d\alpha )((\Phi^{-1})_* J(\Phi_* X), X) = $$

\be
= d\alpha (J(\Phi_*X), \Phi_* X)=0 \; \text{ for any } X \text{ horizontal vector in } \R^5.
\ee

Therefore the induced almost complex structure $\Phi ^*J$ is anti-compatible with the symplectic form $d (\Phi ^* \alpha )= dx^1 d y^1 +  dx^2 d y^2$.

\medskip

It is now clear that we can afford to work in a ball centered at $0$ of the standard contact structure $(\R^5, \zeta)$ with an almost complex structure $J$ such that $d \zeta(v, Jv)=0$.

\medskip

In view of the construction of "positive foliations", we can start with the following question: given a point in $\R^5$ and a $J$-invariant plane through it, can we find an embedded Legendrian disk which is $J$-invariant and has the chosen plane as tangent? 

The following is of fundamental importance:

\begin{oss}
\label{oss:data}
Given a legendrian immersion of a 2-surface in $\R^5$, any tangent plane $D$ to it necessarily satisfies the condition $d \alpha (D)=0$ (see $\cite{R}$). (\ref{eq:anticomp}) is therefore a necessary condition for the local existence of $J$-invariant disks through a point in any chosen direction. 

On the other hand, always from \cite{R}, we know that every Lagrangian in $\R^4$ can be uniquely lifted to a Legendrian in $\R^5$ after having chosen a starting point in $\R^5$.

In this subsection we will prove, in particular, the sufficiency of condition (\ref{eq:anticomp}) for the local existence of a $J$-invariant Legendrian for which we assign its tangent at chosen point.
\end{oss}

\medskip

Write $J(0)=J_0$. Let us analyse the case $J=J_0$ everywhere. Then we can explicitly find an embedded Legendrian disk which is $J_0$-invariant and with tangent at $0$ the given $D$. This goes as follows: $J_0$ can be seen as an almost complex structure on $\R^4$ and the plane $D$ is $J_0$-invariant in $\R^4$, and by the condition $d \alpha(D)=0$ it is lagrangian for the symplectic form $d\alpha$. Therefore, by the result in $\cite{R}$, the plane can be lifted to a legendrian surface $\tilde{D}$ in $\R^5$ passing through $0$. This surface is then trivially $J_0$-invariant and the tangent at $0$ is $D$ since $H_0=\R^4$.

What about the case of a general $J$ with $J(0)=J_0$? We want to use a fixed point argument in order to find a $J$-invariant Legendrian close to $\tilde{D}$. To achieve that, we need to ensure that we are working in a neighbourhood where $J-J_0$ is bounded in a suitable $C^{m, \nu}$-norm.

\medskip

Dilate $\R^5$ about the origin as follows:
\[\Lambda_r: (x_1, y_1, x_2, y_2, t) \to \left( \frac{x_1}{r}, \frac{y_1}{r} , \frac{x_2}{r}, \frac{y_2}{r} ,\frac{t}{r}\right).\]
This dilation changes the contact structure: indeed, pulling back the standard contact form by $\Lambda_r^{-1}$ we get
$$ r^2 \left( \frac{1}{r} dt -(y_1 dx^1 +y_2 dx^2) \right)$$
thus the horizontal hyperplanes are 
$$Span \{\p_{x_1}+ r y_1 \p_t, \p_{x_2}+ r y_2 \p_t, \p_{y_1}, \p_{y_2}\}.$$

The dilation has therefore the effect of "flattening" (with respect to the euclidean geometry) the horizontal distribution\footnote{The dilation $\left( \frac{x_1}{r}, \frac{y_1}{r} , \frac{x_2}{r}, \frac{y_2}{r} ,\frac{t}{r^2}\right)$, on the other hand, would leave the horizontal distribution unchanged. This non-homogeneous transformation would still allow the proof of proposition \ref{Prop:uno}, but in view of propositions \ref{Prop:polar} and \ref{Prop:parallel} it is convenient to work with the "flattened" distribution.}. 

We also pull back by $\Lambda_r$ the almost complex structure $J$ and for $r$ small enough we can ensure that $\|\Lambda_r^* J-J_0\|_{C^{2, \nu}}= r \|J-J_0\|_{C^{2, \nu}(B_r)}$ is as small as we want. 

Finding Legendrians in the dilated contact structure that are invariant for $\Lambda_r^* J$ is the same as finding $J$-invariant Legendrians in $(\R^5, \zeta)$ in a smaller ball around $0$: we can go from the first to the second via $\Lambda_r^{-1}$. It is then enough to work in $(\R^5, (\Lambda_r^{-1}) ^*\zeta)$, with the almost complex structure $\Lambda_r^* J$. 

By abuse of notation, we will drop the pull-backs and forget the factor $r^2$; our assumptions, to summarize, will be as follows:
\be
\label{eq:dilatedcontact}
\begin{array}{c}
\alpha=\left( \frac{1}{r} dt -(y_1 dx^1 +y_2 dx^2) \right) \\
d \alpha= dx^1 d y^1 +  dx^2 d y^2 \\
\|J-J_0\|_{C^{2, \nu}(B_1)}= \eps \; \text{ for an arbitrarily small $\eps$.} \end{array}
\ee 

\medskip

\textbf{Basic example}. 
What can we say about an almost complex structure $J$ on $(\R^5, \zeta)$ such that $d \zeta(v, Jv)=0$ (and $d \zeta(v,w)= -d \zeta(Jv,Jw)$) for all horizontal vectors $v$ and $w$? 

These conditions, applied to the vectors $\p_{x_1}+y_1 \p_t, \p_{x_2}+y_2 \p_t, \p_{y_1}, \p_{y_2}$, together with $J^2=-Id$, give that $J$ must have the following coordinate expression for some smooth functions $\sigma, \beta, \gamma, \delta$ of five coordinates\footnote{We assume here that $\gamma \neq 0$. Remark that $\beta$ and $\gamma$ cannot both be $0$, since $J^2=-Id$.}:

\be
\label{eq:Jexpr}
\left \{ \begin{array}{c} 
J(\p_{x_1}+y_1 \p_t)= \sigma (\p_{x_1}+y_1 \p_t) + \beta(\p_{x_2}+y_2 \p_t) +\gamma \p_{y_2}\\
J(\p_{x_2}+y_2 \p_t)= -\sigma (\p_{x_2}+y_2 \p_t) + \delta (\p_{x_1}+y_1 \p_t) - \gamma \p_{y_1} \\ 
J(\p_{y_1})= \sigma \p_{y_1} + \delta \p_{y_2}+ \frac{1+\sigma^2+ \beta \delta}{\gamma} (\p_{x_2}+y_2 \p_t) \\ 
J(\p_{y_2})= -\sigma \p_{y_2} - \frac{1+\sigma^2+ \beta \delta}{\gamma} (\p_{x_1}+y_1 \p_t) + \beta \p_{y_1} . \end{array} \right .
\ee

\medskip

\textbf{Local existence of $J$-invariant Legendrians}. The results that we are going to prove, in particular the proofs of propositions \ref{Prop:uno}, \ref{Prop:polar} and \ref{Prop:parallel}, follow the same guidelines as the proofs presented in the appendix of \cite{RT1}, with the due changes. In particular, equation (\ref{eq:lapl}) takes the place of equation (A.2) in \cite{RT1}. In the $5$-dimensional contact case that we are addressing, therefore, we will face a second order elliptic problem, in contrast to the $4$-dimensional almost-complex case where the equation was of first order.

With remark \ref{oss:data} in mind, we will see that, in order to produce a solution of our problem, we will need to adapt coordinates to the chosen data, i.e. the point and the direction. Later on, with (\ref{eq:estfpX}) and (\ref{eq:estF}), we will understand the dependence on the data for the solutions obtained. 

At that stage we will be able to produce the key tool for the proof of theorem \ref{thm:main}: foliations made of $3$-dimensional surfaces having the property of intersecting any $J$-invariant Legendrian in a positive way, see the discussion following proposition \ref{Prop:parallel}.

\begin{Prop}
\label{Prop:uno}
Let $(\R^5, \alpha)$ be the contact structure described in (\ref{eq:dilatedcontact}),  with $J$ an almost-complex structure defined on the horizontal distribution $H= Ker \; \alpha$ such that $d \alpha (J v, v)= 0 \; \text{ for any } v \in H$. 

Then, if $\eps$ is small enough\footnote{It will be clear after the proof that $\eps$ must be small compared to $\frac{1}{\|J_0\| N^2 }$, where $N$ is a constant depending on an elliptic operator defined from $J_0$.}, for any $J$-invariant 2-plane $D$ passing through $0$, there exists locally an embedded Legendrian disk which is $J$-invariant and goes through $0$ with tangent $D$.
\end{Prop}

From the discussion above, we can see that we are actually showing the following:

\begin{Prop}
\label{Prop:unoM}
Let $\mathcal{M}$ be a five-dimensional manifold endowed with a contact form $\alpha$ and let $J$ be an almost-complex structure defined on the horizontal distribution $H= Ker \; \alpha$ such that $d \alpha (J v, v)= 0 \; \text{ for any } v \in H$. 

Then at any point $p \in \mathcal{M}$ and for any $J$-invariant 2-plane $D$ in $T_p \mathcal{M}$, there exists an embedded Legendrian disk $\mathcal{L}$ which is $J$-invariant and goes through $p$ with tangent $D$.
\end{Prop}

\begin{proof}[\bf{proof of proposition \ref{Prop:uno}}]

\textit{Step 1}. Before going into the core of the proof, we need to perform a suitable change of coordinates.

The hyperplane $H_0$ coincides with $\R^4=\{t=0\}$. Up to performing an orthogonal rotation of coordinates in $H_0=\R^4$, we can assume that $D=\p_{x_1} \wedge J(\p_{x_1})$. 

Perform another orthogonal change of coordinates in $\R^4=\{t=0\}$ (the coordinate $t$ stays fixed) such that, in the new coordinates, that we still denote $(x_1, y_1, x_2, y_2)$, we have that the symplectic form $d \alpha$ still has the form $dx^1 d y^1 +  dx^2 d y^2$. This can be achieved by any rotation which sends the old quadruplet $\{\p_{x_1}, \p_{y_1}, \p_{x_2}, \p_{y_2}\}$ (based at $0$) to $\{\p_{x_1}, I (\p_{x_1}), W, I W\}$, where $I$ is the standard complex structure\footnote{See (\ref{eq:I}).} associated to $dx^1 d y^1 +  dx^2 d y^2$ and $W$ is a vector orthonormal to $\p_{x_1}$ and $J (\p_{x_1})$ for the metric induced by $dx^1 d y^1 +  dx^2 d y^2$ and $I$.

There is freedom on the choice of $W$; in step 2 we will determine it uniquely by imposing a further condition\footnote{This is needed in view of \textit{Step 4}.}. Before doing this we are going to make the notation less heavy.

This linear change of coordinates has not affected the fact that the hyperplanes $H_{\pi^{-1}(q)}$ are parallel\footnote{See remark \ref{oss:parall}.} in the standard coordinates of $\R^5$. This means that, if we take a vector $\p_{x_i}$ [resp. $\p_{y_i}$] in $\R^4$, with base-point $q \in \R^4$, its lift to a horizontal vector based at any point of the fiber $\pi^{-1}(q)$ has a coordinate expression of the form $\p_{x_i} + K^{x_i} \p_t$ [resp. $\p_{y_i}+ K^{y_i} \p_t$], where $K^{x_i}=K^{x_i}(x_1(q), x_2(q), y_1(q), y_2(q))$ and $K^{y_i}=K^{y_i}(x_1(q), x_2(q), y_1(q), y_2(q))$ are linear funtions of the coordinates of $q$ (they come from the last coordinate change). 

We are interested in the expression for $J$ in a neighbourhood of the origin. $J$ acts on the horizontal vectors $\p_{x_i} + K^{x_i} \p_t$, $\p_{y_i}+ K^{y_i} \p_t$. However, since the functions $K^{x_i}$, $K^{y_i}$ are independent of $t$, by abuse of notation we will forget about the $\p_t$-components of the horizontal lifts and speak of the action of $J$ on $\p_{x_i}, \p_{y_i}$, keeping in mind that the coefficients of the linear map $J$ are not constant along a fiber, i.e. $J$ cannot be projected onto $\R^4$.

With this in mind, recalling (\ref{eq:Jexpr}), the expression for $J$ in the unit ball $B^1(0) \subset \R^5$ is as follows: there are smooth functions $\sigma, \beta, \gamma, \delta$ depending on the five coordinates of the chosen point, such that\footnote{In (\ref{eq:J}) we are assuming that $\gamma \neq 0$. This is not restrictive. We can assume to be working in an open set where at least one of the functions $\beta$ and $\gamma$ is everywhere non-zero. If this is the case for $\beta$ and not for $\gamma$, a change of coordinates sending $\p_{x_2} \to \p_{y_2}$ and $\p_{y_2} \to -\p_{x_2}$ would lead us to (\ref{eq:J}) again.} 

\begin{equation}
\label{eq:J}
\left \{ \begin{array}{c} 
J(\p_{x_1})= \sigma \p_{x_1} + \beta\p_{x_2} +\gamma \p_{y_2}\\
J(\p_{x_2})= -\sigma \p_{x_2} + \delta \p_{x_1} - \gamma \p_{y_1} \\ 
J(\p_{y_1})= \sigma \p_{y_1} + \delta \p_{y_2}+ \frac{1+\sigma^2+ \beta \delta}{\gamma} \p_{x_2} \\ 
J(\p_{y_2})= -\sigma \p_{y_2} - \frac{1+\sigma^2+ \beta \delta}{\gamma} \p_{x_1} + \beta \p_{y_1} . \end{array} \right .
\end{equation}

\medskip

\textit{Step 2}. Denote the values of these coefficients at $0$ by $\delta(0)=\delta_0$, $\beta(0)=\beta_0$, $\sigma(0)=\sigma_0$, $\gamma(0)=\gamma_0$. We take now coordinates, that we underline to distinguish them from the old ones, determined by the transformation

\begin{equation}
\label{eq:chcoord}
\left \{ \begin{array}{c} 
\underline{\p_{x_1}}= \p_{x_1}\\
\underline{\p_{y_1}}= \p_{y_1} \\ 
\underline{\p_{x_2}}= \frac{\gamma_0}{\sqrt{\beta_0^2 + \gamma_0^2}} \p_{x_2} - \frac{\beta_0}{\sqrt{\beta_0^2 + \gamma_0^2}}\p_{y_2} \\ 
\underline{\p_{y_2}}= \frac{\beta_0}{\sqrt{\beta_0^2 + \gamma_0^2}} \p_{x_2} + \frac{\gamma_0}{\sqrt{\beta_0^2 + \gamma_0^2}}\p_{y_2} \\
\underline{\p_{t}}= \p_{t}. \end{array} \right .
\end{equation}

In the new coordinates, the endomorphism $J$ at $0$ acts on $\underline{\p_{x_1}}$ as 
$$ J_0(\underline{\p_{x_1}})=  \sigma_0 \underline{\p_{x_1}} + \sqrt{\beta_0^2 + \gamma_0^2}\underline{\p_{y_2}}$$
and the symplectic form $d \alpha$ still has the standard expression.

From now on we will write these new coordinates again as $(x_1, y_1, x_2, y_2, t)$, without underlining them. To summarize what we did in steps 1 and 2: we will now work in the unit ball of $\R^5$ with the symplectic form $ d \alpha= dx^1 d y^1 +  dx^2 d y^2$ on the horizontal distribution and an almost complex structure $J$ such that $\|J-J_0\|_{C^{2, \nu}}< \eps$ for some small $\eps$ that we will determine precisely later on, and such that $J$ is expressed by (\ref{eq:J}) with smooth functions $\sigma, \beta, \gamma, \delta$ depending on the five coordinates and satisfying $\beta_0=0$, $\gamma_0 >0$. The $J$-invariant plane $D$ is given by $D=\p_{x_1} \wedge J(\p_{x_1})$. 

The double coordinate change in steps 1 and 2 can be characterized as the unique change of coordinates such that $d \alpha=dx^1 d y^1 +  dx^2 d y^2$ and $D=\p_{x_1} \wedge J(\p_{x_1})= \gamma_0 (\p_{x_1} \wedge\p_{y_2})$ (for a positive $\gamma_0$).

\medskip

\textit{Step 3}. We are looking for an embedded Legendrian disk with tangent $D$ at the origin, therefore we will seek a Legendrian which is a graph over $D=\p_{x_1} \wedge \p_{y_2}$. Recall from $\cite{R}$ that the projection of any Legendrian immersion in $\R^5$ is a Lagrangian in $\R^4$ with respect to the symplectic form $d\alpha$. A Lagrangian graph over $D=\p_{x_1} \wedge \p_{y_2}$ must be of the form (see III.2 of \cite{HL}, in particular lemma 2.2)
\be
\left(x_1, \frac{\p f(x_1, y_2)}{\p x_1}, -\frac{\p f(x_1, y_2)}{\p y_2}, y_2 \right) \;\text{ for some } f:D^2_{x_1, y_2} \to \R.
\ee

The minus in the $x_2$-component is due to the fact that $I(\p_{y_2})=-\p_{x_2}$, while $I(\p_{x_1})=\p_{y_1}$. Our problem can be now restated as follows: find a function $f:D^2 \to \R$ such that the lift $\mathcal{L}$ with starting point $0$ of the Lagrangian disk $$L(x_1, y_2):=\left(x_1, \frac{\p f}{\p x_1}, -\frac{\p f}{\p y_2}, y_2 \right)$$ is $J$-invariant\footnote{By lift of $L$ with starting point $0$, we mean that the $t$-component of $\mathcal{L}(0,0)$ is $0$.

At this point we can see how, in the $5$-dimensional contact case, the necessity of lifting naturally leads to a second-order equation. In the $4$-dimensional almost-complex case, one does not need to worry about lifting.}.

The $J$-invariance condition is a constraint on the tangent planes: it is expressed by the following equation for the lift $\mathcal{L}$ of $L$:
\be
\label{eq:Jinvplane}
J\left(\frac{\p \mathcal{L}}{\p x_1}\right)=(1+\lambda) \frac{\p \mathcal{L}}{\p y_2} + \mu \frac{\p \mathcal{L}}{\p x_1}.
\ee

However, thanks to what we observed in step 1, the tangent vectors $\frac{\p \mathcal{L}}{\p x_1}$ and $\frac{\p \mathcal{L}}{\p y_2}$ to the lift $\mathcal{L}$ at any point have the first four components which equal $\frac{\p L}{\p x_1}$ and $\frac{\p L}{\p y_2}$ at the projection of the chosen point, independently of where we are lifting along the fiber; the fifth component of $\frac{\p \mathcal{L}}{\p x_1}$ and $\frac{\p \mathcal{L}}{\p y_2}$ is uniquely determined by the other four and by the point $(x_1, y_1, x_2, y_2)$ in $\R^4$. We will therefore consider equation (\ref{eq:Jinvplane}) only for $\frac{\p L}{\p x_1}$ and $\frac{\p L}{\p y_2}$. 

We denote the partial derivatives $\frac{\p f(x_1, y_2)}{\p x_1}$, $\frac{\p f(x_1, y_2)}{\p y_2}$, $\frac{\p^2 f(x_1, y_2)}{\p x_1^2}$, $\frac{\p^2 f(x_1, y_2)}{\p x_1 \p y_2}$ and $\frac{\p^2 f(x_1, y_2)}{\p y_2^2}$ respectively by $f_1$, $f_2$, $f_{11}$, $f_{1_2}$ and $f_{22}$. Then 
\be
\label{eq:lagrdisk}
L(x_1, y_2)=\left(\begin{array}{rrrr} 
x_1 \\
f_1 \\
-f_2 \\
y_2 \end{array} \right), \;\;
\frac{\p L}{\p x_1}=\left( \begin{array}{rrrr}
1\\ 
f_{11}\\ 
-f_{12}\\ 
0 \end{array} \right), \;\; \frac{\p L}{\p y_2}=\left( \begin{array}{rrrr}
0\\ 
f_{12}\\ 
-f_{22}\\ 
1 \end{array} \right).
\ee

It should be however born in mind that $J$ does depend on where we are lifting! After little manipulation, making use of (\ref{eq:J}), the equation in (\ref{eq:Jinvplane}) reads:

\begin{equation}
\label{eq:Jinvplane3}
 \left( \begin{array}{rrrr}
\sigma - \delta f_{12} - \mu\\ 
\sigma f_{11}- (1-\gamma) f_{12} - \lambda f_{12} - \mu f_{11}\\ 
\Delta f+\beta + \frac{1+\sigma^2 + \beta \delta -\gamma}{\gamma} f_{11} + \sigma f_{12} +\lambda f_{22} + \mu f_{12}\\ 
\gamma-1+\delta f_{11} -\lambda \end{array}\right)=\left( \begin{array}{rrrr}
0\\ 
0\\ 
0\\ 
0 \end{array} \right) ,
\end{equation}

with $\sigma, \beta, \gamma, \delta$ evaluated at the lift of $L=(x_1, f_1, -f_2, y_2 )$ in $\R^5$ with starting point $0$.

From the first and fourth line of (\ref{eq:Jinvplane3}) we get 
\be
\label{eq:mulambda}
\mu= -\delta f_{12}+\sigma, \;\; \lambda=\gamma-1+\delta f_{11}.
\ee

The second line of (\ref{eq:Jinvplane3}) can be checked to hold automatically true with these values of $\mu$ and $\lambda$. Then we need to find $f$ solving the third line of (\ref{eq:Jinvplane3}) with the $\mu$ and $\lambda$ given in (\ref{eq:mulambda}). We stress once again that (\ref{eq:Jinvplane3}) should be solved for $f$ with $\sigma, \beta, \gamma, \delta$ depending on the lift of $(x_1, f_1, -f_2, y_2 )$. Let us write the third line of (\ref{eq:Jinvplane3}) explicitly. It reads
\be
\label{eq:lapl}
\sum_{i,j=1}^2 M_{ij}  f_{ij} = \delta (f_{12}^2-f_{11}f_{22}) -\beta +\sum_{i,j=1}^2 A_{ij} f_{ij},
\ee
where $M$ and $A$ are the matrices 

\be
\label{eq:MA}
M=\left( \begin{array}{cc}
\frac{1+\sigma_0^2}{\gamma_0} & -\sigma_0\\

-\sigma_0 & \gamma_0 \end{array} \right),\;\; A= \left( \begin{array}{cc}
\frac{1+\sigma^2+\beta \delta}{\gamma}-\frac{1+\sigma_0^2}{\gamma_0} & \sigma-\sigma_0\\

\sigma-\sigma_0 & \gamma_0-\gamma \end{array} \right).
\ee
$M$ is a positive definite matrix and satisfies, for any vector $(\xi_1, \xi_2) \in \R^2$, the ellipticity condition 
$$\sum_{i,j=1}^2 M_{ij}\xi_i \xi_j  \geq k( \xi_1^2 + \xi_2^2)   \text{ for a positive }k.$$
 Remark also that, at the origin, $\beta(0)=0$ and $A$ is the zero matrix. The zero function $f=0$, describes the disk $D$. We want to solve equation (\ref{eq:lapl}) by a fixed point method in order to find a solution $f$ close to $0$. We will write $M f$ for the elliptic operator on the left hand side of (\ref{eq:lapl}).

Consider the functional $\mathcal{F}$ defined as follows: for $h\in C^{2,\nu}$ let $\mathcal{F}(h)$ be the solution of the following well-posed elliptic problem:

\be
\left\{ \begin{array}{rr} M [\mathcal{F}(h)]= \delta_h (h_{12}^2 - h_{11}h_{22}) - \beta_h + {A_{ij}}_h h_{ij}   \\
 \mathcal{F}(h)\left|_{\p D^2} \right. = 0   \end{array} \right.
\ee
where by $\delta_h$, $\beta_h$ and ${A_{ij}}_h$ we mean respectively the functions $\delta$, $\beta$ and $A_{ij}$ evaluated at the lift\footnote{As always, we are lifting the point $(0,h_1(0,0), -h_2(0,0),0)\in \R^4$ to the point $(0,h_1(0,0), -h_2(0,0),0,0)\in \R^5$. This determines the lift of $(x_1, h_1, -h_2, y_2)$ uniquely.} of $(x_1, h_1, -h_2, y_2)$ in $\R^5$ and considered as functions of $(x_1, y_2)$. A fixed point of $\mathcal{F}$ is a solution of (\ref{eq:lapl}).
We know from elliptic regularity that $\mathcal{F}(h)$ belongs to the space $C^{2,\nu}$ and Schauder estimates give
\be
\label{eq:schauder}
\|\mathcal{F}(h)\|_{C^{2,\nu}} \leq N \|\delta_h (h_{12}^2 - h_{11}h_{22}) - \beta_h + A_{ij} h_{ij} \|_{C^{0,\nu}}
\ee
for an universal constant $N$ (depending on $k$). To make the notation simpler in the following, we will assume $N>2$.

We are about to show the following claim: for $\|J-J_0\|_{C^{2,\nu}}$ small enough, the functional $\mathcal{F}$ is a contraction from the closed ball 

\be
\label{eq:littleball}
\left\{h \in C^{2,\nu}: \|h\|_{C^{2,\nu}} \leq \frac{1}{48 \max\{1,|\delta_0|\}N}\right\}
\ee

into itself. 

First of all, let us compute, for $h,g \in C^{2,\nu}$,

$$
M[\mathcal{F}(h)-\mathcal{F}(g)] = \delta_h (h_{12}^2 -g_{12}^2+ g_{11}g_{22}- h_{11}h_{22}) +
$$
$$
(g_{12}^2- g_{11}g_{22})(\delta_h-\delta_g)+ (\beta_g- \beta_h) + {A_{ij}}_h h_{ij} -{A_{ij}}_g g_{ij}=
$$
\be
\label{eq:diff}
=\delta_h[(h_{12}+g_{12})(h_{12}-g_{12}) + g_{11}(g_{22}-h_{22})+ h_{22}(g_{11}-h_{11})]+
\ee
$$
+(g_{12}^2- g_{11}g_{22})(\delta_h-\delta_g)+ (\beta_g- \beta_h) + {A_{ij}}_h (h_{ij}-g_{ij})+({A_{ij}}_h -{A_{ij}}_g) g_{ij}.
$$ 
Remark that we have bounds of the form
\be
\left\{ \begin{array}{rr} \|\delta_h\|_{C^1} \leq \|\delta\|_{C^0} +2 \| \nabla \delta\|_{C^0} \|h\|_{C^2}, \\
\|\delta_h-\delta_g\|_{C^1} \leq 2( \|\delta\|_{C^2} \|h\|_{C^2} + 2 \| \nabla \delta\|_{C^0}) \|h-g\|_{C^2}, \end{array} \right.
\ee

where the norms are taken in the unit ball $B^5_1(0)$. Similar bounds hold true for $\beta$ and $A_{ij}$.

For $\|J-J_0\|_{C^{2,\nu}}$ small enough, in particular if $\|\delta\|_{C^2} \leq 2|\delta_0|$, Schauder theory applied to equation (\ref{eq:diff}) with boundary data $(\mathcal{F}(h)-\mathcal{F}(g))\left|_{\p D^2}\right.  = 0$ gives 
\be
\label{eq:contr}
\|\mathcal{F}(h)-\mathcal{F}(g)\|_{C^{2,\nu}} \leq$$ 
$$  N\left( 4|\delta_0|(\|h\|_{C^{2,\nu}}+\|g\|_{C^{2,\nu}} ) + 4\|g\|_{C^{2,\nu}}^2 + 4\|\beta\|_{C^2} + 6\|A\|_{C^2}\right) \|h-g\|_{C^{2,\nu}} .
\ee

Let us now estimate, again by (\ref{eq:schauder})
\be
\label{eq:into}
\|\mathcal{F}(h)\|_{C^{2,\nu}} \leq 4 N|\delta_0| \|h\|_{C^{2,\nu}}^2 +2N \|A\|_{C^1} \|h\|_{C^{2,\nu}} +\|\beta\|_{C^1}.
\ee

If $\|\beta\|_{C^2} +\|A\|_{C^2} \leq \frac{1}{24 \max\{1,|\delta_0|\} N^2}$, which surely holds for $\|J-J_0\|_{C^{2,\nu}}$ small enough, by (\ref{eq:contr}) and (\ref{eq:into}) we get that $\mathcal{F}$ is a contraction of the forementioned ball (\ref{eq:littleball}). By Banach-Caccioppoli's theorem, there exists a unique fixed point $f$ of $\mathcal{F}$, so we get a solution to equation (\ref{eq:lapl}) of small $C^{2,\nu}$-norm. 

More precisely, from (\ref{eq:into}) we get
\be
\label{eq:boundJ}
\|f\|_{C^{2,\nu}}\leq K(\eps)
\ee

where $K(\eps)$ is a constant which goes to zero as $\eps \to 0$.

The lift of $(x_1, f_1, -f_2, y_2)$ is an embedded, $J$-invariant, Legendrian disk that we denote $\mathcal{L}_{0,D}$. This disk, however, does not necessarily pass through the origin. 

\medskip

\textit{Step 4}. In step 3 we constructed a $J$-invariant disk which is a small $C^{2,\nu}$-perturbation of $D$ but that might not pass through $0$.

We need to generalize the construction performed in step 3. Let us set up notations: we are working in the unit ball of $\R^5$, with coordinates $(x_1,y_1,x_2, y_2,t)$, such that the point $p$ in the statement of Proposition \ref{Prop:uno} is the origin $0$ and $D$ is the plane $\p_{x_1}\wedge J_P(\p_{x_1})$ at the origin. The almost complex structure $J$ satisfies $\|J-J_0\|_{C^{2,\nu}}< \eps$ for some positive $\eps$ as small as we want. An upper bound for $\eps$ was described in step 3.

Denote by $Z_r:=\{(0,y_1,x_2, 0,t): x_2^2+y_1^2 \leq r^2, \; |t|\leq r \}$. For any point $P \in B_1(0) \subset \R^5$, the set of $J$-invariant planes at $P$ can be parametrized by $\CP^1$: we will use the following identification between $H_P$ and $\C^2$

\be
\label{eq:ident}
\p_{x_1}= (0,1), \; J_P(\p_{x_1})= (0,i), \; \p_{y_1}= (1,0), \; J_P(\p_{y_1})= (i,0).
\ee

Passing to the quotient, we get a pointwise identification $\eta_P$ between $\{\Pi: \Pi \text{ is a $J$-invariant 2-plane in } H_P\}$ and $\CP^1$. In this identification, for any point $P$ the planes $\p_{x_1}\wedge J_P(\p_{x_1})$ are represented by $[0,1] \in \CP^1$.

We denote by $\mathcal{U}^P_r$ the set of $J$-invariant planes at $P$ which are identified via $\eta_P$ with $\mathcal{U}_r:=\{[W_1, W_2] \in \CP^1: |W_1|\leq r |W_2|\}$. This allows us to regard the set

$$\{(P,X) \text{ with } P\in Z_r \text{ and } X \in \mathcal{U}^P_r\}$$
as the product manifold 
$$Z_r \times \mathcal{U}_r.$$

For any couple $(P,X) \in Z_1 \times \mathcal{U}_1$, we can set coordinates adapted to $(P,X)$ as follows: after a translation sending $0$ to $P$, we can rotate the coordinate axis by choosing $v_X$, the orthogonal projection of $\p_{x_1}$ onto the closed unit ball in $X$ and setting the new $\p_{x_1}$ to be $\frac{v_X}{|v_X|}$. With this choice, we can perform the same change of coordinate\footnote{In the sequel we will denote by $\E_{P,X}$ the affine map which induces this change of coordinates.} that we had in steps 1,2 and 3.

Now, using a fixed point argument as in step 3, we can associate to any couple $(P,X) \in Z_1 \times \mathcal{U}_1$ a $J$-invariant disk that we denote by $\mathcal{L}_{P,X}$. 

The estimate given by (\ref{eq:boundJ}) implies that $|T\mathcal{L}_{P,X}-X| \leq K(\eps)$,  so in particular we have $T \mathcal{L}_{P,X} \in \mathcal{U}_{1+K(\eps)}$.

Hence $\mathcal{L}_{P,X}$ is transversal to the 3-dimensional plane $\{(0,y_1,x_2, 0,t)\}$. Consider the point $Q:=\mathcal{L}_{P,X} \cap \{(0,y_1,x_2, 0,t)\}$ and the tangent plane to $\mathcal{L}_{P,X}$ at $Q$. We get a map 
$$\Psi: Z_1 \times \mathcal{U}_1 \to Z_{1+K(\eps)} \times \mathcal{U}_{1+K(\eps)}$$
$$\Psi(P,X) = (Q, T_Q \mathcal{L}_{P,X}).$$

Condition (\ref{eq:boundJ}) tells us that 

\be
\label{eq:perturbpsi}
\|\Psi-Id\|_{C^{2,\nu}} \leq K(\eps)
\ee

where $K(\eps)\to 0$ as $\eps \to 0$. Therefore $\Psi$ is invertible on an open set $\mathcal{U}_r$, and $\mathcal{U}_r$ is $\eps$-close to $\mathcal{U}_1$. So, for $\|J-J_0\|_{C^{2,\nu}}=\eps$ small enough, by inverting $\Psi$ we get that for every point $Q$ in $Z_r$ and any $J$-invariant disk $Y$ through $Q$ lying in $\mathcal{U}^q_r$, we can find a couple $(P,X) \in Z_1 \times \mathcal{U}_1$ such that $\mathcal{L}_{P,X}$ goes through $Q$ with tangent $Y$.

In particular we can find an embedded, $J$-invariant Legendrian disk which goes through $0$ with tangent $\p_{x_1} \wedge J\p_{x_1}$.

Remark that, due to the smoothness of $J$, the same proof performed using the space $C^{m, \nu}$ for any $m \geq 2$ rather than $C^{2, \nu}$ gives that the disks $\mathcal{L}_{P,X}$ are $C^{\infty}$-smooth.

We have thus proved Proposition \ref{Prop:uno}. 

\end{proof}

Remark again that we have actually shown more: in the coordinates described in (\ref{eq:ident}), for each couple $(p, X)$, $p \in Z_1$, $X \in \mathcal{U}_1 \subset \CP^1$ we can find \footnote{The above proof actually yielded the result for an open set $\mathcal{U}_r$ with $r$ close to $1$, but of course we can assume that it holds for $r=1$.} an embedded, $J$-invariant Legendrian disk which goes through $p$ with tangent $X$.

This will be useful for the next results.

\medskip
\textbf{Dependence on the choice of coordinates}.
In the previous proof we constructed, from each couple $(p, X)$, $p \in Z_1$, $X \in \mathcal{U}_1 \subset \CP^1$, a disk $\mathcal{L}_{p,X}$ whose projection $L_{p,X}$ in $\R^4$ is described, in suitable coordinates for which $X= \p_{x_1} \wedge \p_{y_2}$, as a graph $(x_1, f_1, -f_2, y_2)$. To make notations adapted to what we want to develop in this section, we will write $f^{p,X}$ instead of $f$ for the function whose gradient describes the graph. 

Given $(p,X)$, in \textit{step 4} we chose uniquely the change of coordinates to perform in order to write the equations that lead to the solution $f^{p,X}$ of (\ref{eq:lapl}). We denote the affine map that induces the change of coordinates by $\E_{p,X}$. The function $f^{p,X}(x_1, y_2)$ solves equation (\ref{eq:lapl}) with coefficients $\delta, \beta, \sigma, \gamma$ depending on $\E_{p,X}$, therefore we will now write it as

\be
\label{eq:lapl2}
M^{p,X}_{ij} f^{p,X}_{ij} = \delta^{p,X} \left((f^{p,X})_{12}^2-(f^{p,X})_{11}(f^{p,X})_{22}\right) -\beta^{p,X} +\sum_{i,j=1}^2 A^{p,X}_{ij} (f^{p,X})_{ij},
\ee

where $M^{p,X}$ and $A^{p,X}$ are as in (\ref{eq:MA}) but we explicited the $(p,X)$-depen\-dence. All the functions in (\ref{eq:lapl2}) are functions of $(x_1, y_2)$, but we want to see how the solution $f^{p,X}(x_1, y_2)$ changes with $(p,X)$. In this section we will denote by $\nabla_X$ and $\nabla_p$ the gradients with respect to the variables $X \in \mathcal{U}_1$ and $p \in Z_1$. The $x_1$ and $y_2$ derivatives will still be denoted by pedices $i,j \in \{1,2\}$.

\begin{lem}
As $X \in \mathcal{U}_1$ and $p\in Z_1 \subset \R^5$, the solutions $f^{p,X}$ of the corresponding equations (\ref{eq:lapl2}) satisfy
\be
\label{eq:estfpX}
\text{for } \; s, l \in \{0,1,2\}, \;\;  \| \nabla^s_p \nabla^l_X f^{p,X}\|_{C^{2,\nu}} \leq K (\eps).
\ee
where $K (\eps)$ is a constant which goes to $0$ as $\eps \to 0$ (so we can make $K (\eps)$ as small as we want be dilating enough).
\end{lem}

\begin{proof}
Differentiating (\ref{eq:lapl2}) w.r.t. $X$

$$
M^{p,X}_{ij} (\nabla_X f^{p,X})_{ij} = (\nabla_X \delta^{p,X}) \left((f^{p,X})_{12}^2-(f^{p,X})_{11}(f^{p,X})_{22}\right) +  $$ $$+\delta^{p,X} 
\left(2(f^{p,X})_{12}(\nabla_X f^{p,X})_{12}-(f^{p,X})_{11}(\nabla_X f^{p,X})_{22}-(f^{p,X})_{22}(\nabla_X f^{p,X})_{11} \right) -$$
$$ - \nabla_X \beta^{p,X} + (\nabla_X A^{p,X})_{ij} (f^{p,X})_{ij} + A^{p,X}_{ij} (\nabla_X f^{p,X})_{ij} - (\nabla_X M^{p,X})_{ij} (f^{p,X})_{ij}.$$

The quantities $\nabla_X \delta^{p,X}$, $\nabla_X M^{p,X}$, etc, are all bounded in $C^{2,\nu}$-norm by some constant $K$ (uniform in $p$ and $X$) which depends on $\|J\|_{C^{2,\nu}}$ and $\|\E\|_{C^{2,\nu}}$. 

Recalling that $\|f^{p,X}\|_{C^{2,\nu}} \leq K(\eps)$, by elliptic theory we get that $\nabla_X f^{p,X}$ satisfies

$$\| \nabla_X f^{p,X}\|_{C^{2,\nu}} \leq K (\eps) + \| \nabla_X \beta^{p,X}\|_{C^{0,\nu}}.$$

$E_{p,X}$ was chosen so that the function $\beta^{p,X}(x_1, y_2)$ satisfies $\beta^{p,X}(0,0)=0$ for all $(p,X)$. Therefore $$\text{for $s,l \in \{0,1,2\}$ } \; \nabla^s_p \nabla^l_X \beta^{p,X} =0 \text{ when evaluated at } (x_1, y_2)=(0,0).$$

Then it is not difficult to see that 

$$\text{for } \; s,l \in \{0,1,2\}, \;\; \| \nabla^s_p \nabla^l_X \beta^{p,X}\|_{C^{1}} \leq K(\eps)  \text{ for } (x_1, y_2)\in D^2.$$

Therefore 
$$\| \nabla_X f^{p,X}\|_{C^{2,\nu}} \leq K (\eps).$$

In the same way we can get estimates of the form 
$$
\text{for } \; s,l \in \{0,1,2\}, \;\; \| \nabla^s_p \nabla^l_X f^{p,X}\|_{C^{2,\nu}} \leq K (\eps).
$$

\end{proof}

\medskip

\textbf{Legendrians as graphs on the same disk}. For each couple $(p,X) \in Z_1 \times \mathcal{U}_1$, we have that the embedded disk $\mathcal{L}_{\Psi^{-1}(p,X)}$ passes through $p$ with tangent $X$. 

So far, each $f^{p,X}$ was produced in the system of coordinates induced by $\E_{p,X}$, so $L_{p,X}$ was seen as a graph on $X$. However, thanks to (\ref{eq:boundJ}), $L_{p,X}$ is also a $C^{2, \nu}$-graph over $[0,1]$ for any $X\in \mathcal{U}_1$. We will now look at all $X\in \mathcal{U}_1$ and at all $L_{p,X}$ as graphs on $[0,1]$. In particular we will concentrate on the planes $X$ through points $(0, t)\in \R^4 \times \R$ and on $\mathcal{L}_{\Psi^{-1}((0,t),X)}$, the $J$-invariant Legendrian which goes through $(0,t)$ with tangent $X$.

Any $X\in \mathcal{U}_1^{(0,t)}$, which is a $J$-invariant 2-plane through $(0,t)$, is described as the graph over $\p_{x_1} \wedge \p_{y_2}\cong[0,1]$ of an affine $\R^2$-valued function
$$ H^X :(x_1, y_2) \rightarrow  (h^X_1, -h^X_2).$$
If $t=0$, we can use complex notation, identifying $H_0=\R^4$ with $\C^2$ as in (\ref{eq:ident}), so
\be
\label{eq:complexcoords}
\p_{x_1}= (0,1), \; J_0(\p_{x_1})= (0,i), \; \p_{y_1}= (1,0), \; J_0(\p_{y_1})= (i,0).
\ee
Then, if $X=[W_1, W_2]$, we have that $H^X$ can be expressed as
$$ H^X :z \rightarrow  \zeta=\frac{W_1}{W_2} z. $$

Otherwise, if $t \neq 0$, $H^X$ is just an affine function since $J_{(0,t)} \neq J_0$ in general.

\medskip

What about $L_{\Psi^{-1}((0,t),X)}$, the projection of $\mathcal{L}_{\Psi^{-1}((0,t),X)}$ onto $\R^4$? It was described as the graph of the function $f^{\Psi^{-1}((0,t),X)}$ over the unit disk in the 2-plane given by the second component of $\Psi^{-1}((0,t),X)$. 

Of course, if we want to write it as a graph on $\p_{x_1} \wedge \p_{y_2}$, we will only be able to do so on a restricted disk, for example $\{(x_1, y_2): |x_1^2 + y_2 ^2| \leq \frac{1}{2}\}$. To simplify the exposition, however, we will assume that $f^{\Psi^{-1}((0,t),X)}$ was defined on a larger disk $D_X$ inside $X$ so that, for any $X\in \mathcal{U}_1$, $L_{\Psi^{-1}((0,t),X)}$ can be written as a graph on the unit disk $\{(x_1, y_2): |x_1^2 + y_2 ^2| \leq 1\}$ in the $\p_{x_1} \wedge \p_{y_2}$-plane.

We will denote by $D_0$ the 2-dimensional unit disk, and we will identify it with $\{(x_1, y_2): |x_1^2 + y_2 ^2| \leq 1\}$ in the $\p_{x_1} \wedge \p_{y_2}$-plane.

\medskip

It is not difficult to see that, for each choice of $t$ and $X$, there are a diffeomorphism $d$ from $D_0$ to the enlarged disk $D_X$ and an affine transformation $\T$ of $\R^2$ depending on $X$, $t$ and $\Psi^{-1}((0,t),X)$ such that, over $D_0$, $L_{\Psi^{-1}((0,t),X)}$ is the graph of a function of the form
\be
\label{eq:defF}
H^{t,X} + F^{t,X}: D_0 \rightarrow \R^2, \text{ with } F^{t,X}:= \T \circ f^{\Psi^{-1}((0,t),X)} \circ d.
\ee 

Both $d$ and $\T$, due to the estimate (\ref{eq:perturbpsi}), have bounded derivatives 
\be
\label{eq:estdT}
n,s,l \in \{0,1,2\}, \;\; \|\nabla_z^n \nabla_p^s \nabla_X^l \T\|_{L^\infty} + \|\nabla_z^n \nabla_p^s \nabla_X^l d\|_{L^\infty} \leq K
\ee
uniformly in $X \in \mathcal{U}_1$, $p \in Z_1$ and $z \in D_0$.

For $X\in \mathcal{U}_1$, from the definition (\ref{eq:defF}), using (\ref{eq:estdT}), (\ref{eq:estfpX}) and (\ref{eq:perturbpsi}), we get, for $n,s,l \in \{0,1,2\}$,

\be
\label{eq:estF}
\| \nabla_z^n \nabla_p^s \nabla_X^l F^{t,X}\|_{L^\infty} \leq K \| \nabla_z^n \nabla_p^s \nabla_X^l f^{\Psi^{-1}((0,t),X)}\|_{L^\infty}  \leq K (\eps),
\ee
with $K(\eps) \to 0$ as $\eps \to 0$.

\medskip

\textbf{Construction of the 3-dimensional surfaces: polar foliation}. Using coordinates as in (\ref{eq:complexcoords}), so that the hyperplane $H_0$ is identified with $\C^2$, we expressed each $L_{\Psi^{-1}((0,t),X)}$ as the graph of the following function
$$H^{t,X} + F^{t,X}: D_0 \rightarrow \R^2=\C,$$
which is a perturbation of the affine function $H^{t,X}$ representing the projection on $\R^4$ of the disk $X$ through $(0,t)$.

For the construction that we are about to make, we need to fix a smooth determination of vectors $V_X \in X$ for $X \in \mathcal{U}_1$. There are many ways to do so, we will do it as follows. In our coordinates $\p_{x_1} \in [0,1]$. Then, in the unit disk centered at $0$ inside $X$, chose the vector $v_X$ that minimizes\footnote{There is no geometric meaning in this particular choice, we are just suggesting a smooth determination of vectors, any choice would work the same.} the distance to $\p_{x_1}$ and take $V_X=\frac{v_X}{|v_X|}$.

For any $t \in (-1, 1)$, and for each $X \in \mathcal{U}_1$, at the point $(0, t) \in \R^5$ (here $0 \in \R^4$), take the 2-plane given by $X^t:=V_X \wedge J_{(0,t)}(V_X)$. In this notation, $X=X^0$. For each $X$ and $t$, consider the Legendrian $\mathcal{L}_{\Psi^{-1}((0,t), X^t)}$ going through the point $(0,t)$ with tangent $X^t$: we will now denote it by $\tilde{\mathcal{L}}_{t,X^t}$. As $t \in (-1, 1)$, the union 
\be
\label{eq:sigma}
\Sigma_0^X:= \cup_{t \in (-1, 1)} \tilde{\mathcal{L}}_{t,X^t}
\ee
gives rise to a 3-dimensional smooth surface, as can be seen by writing the parametrizaton of $\Sigma_0^X$ on $D_0 \times (-1,1)$ and using (\ref{eq:estF}) \footnote{Actually, from (\ref{eq:estF}) we get that $\Sigma_0^X$ is $C^2$-smooth. However, (\ref{eq:estfpX}) and (\ref{eq:estF}) can be proved in the same way for higher-order derivatives, so we can get that $\Sigma_0^X$ are as smooth as we want.}.

Each $\tilde{\mathcal{L}}_{t,X^t}$ has a projection $\tilde{L}_{0,X^t}$ onto $\R^4$ which has a representation as the graph on $D_0$ of the function
$$H^{X^t} + F^{X^t}: D_0 \rightarrow \R^2=\C.$$
From $\tilde{L}_{0,X^t}$, the surface $\tilde{\mathcal{L}}_{t,X^t}$ is uniquely recovered by lifting with starting point $(0,t)$.

\medskip

Now with a little more effort we can show:

\begin{Prop}
\label{Prop:polar}
For $X\in \mathcal{U}_1$, the 3-surfaces $\Sigma_0^X$ foliate the set $\{(\zeta, z, t): |\zeta|\leq |z| \leq 1, |t| \leq \frac{1}{2}\} \subset \C \times \C \times \R=\R^5$. 
\end{Prop}

\begin{oss}
Following the terminology used in \cite{BR}, we can restate this proposition by saying that there exist locally \textit{polar} foliations made of $3$-surfaces built from embedded, Legendrian, $J$-invariant disks.
\end{oss}

\begin{proof}
Choose any point $q=(\zeta_q, z_q, t_q) \in B^5 \subset \R^5=\C \times \C \times \R$ which lies inside the set $\{|\zeta|\leq |z|\leq 1, |t| \leq \frac{1}{2}\}$. We need to show the existence and uniqueness of $X\in \mathcal{U}_1$ such that $q \in \Sigma_0^X$.

For $X \in \mathcal{U}_1$, denote by $Q=Q(q,X)$ the intersection point 
\be
Q=Q(q,X):= \Sigma_0^X \cap \{(\zeta, z, t): z=z_q, t=t_q\}.
\ee

This is well-defined because $|T \Sigma_0^X - X| \leq K(\eps)$ and the $3$-plane spanned by $X$ and $\p_t$ is transversal to the $2$-plane $\{(\zeta, z, t): z=z_q, t=t_q\}$ and they have a unique intersection point. By intersection theory, for $\eps$ small enough, $Q$ is well defined for all $X \in \mathcal{U}_1$.

Consider the map

\be
\begin{array}{ccc}
\chi_q: \mathcal{U}_1 & \rightarrow & \CP^1 \\
X & \rightarrow & [\zeta_Q, z_q]
\end{array}
\ee

Due to the structure of $\Sigma_0^X$, the intersection $Q$ is actually realized, for a certain $t$, as 
\be
Q= \tilde{\mathcal{L}}_{0,X^t} \cap \{(\zeta, z, t): z=z_q, t=t_q\}
\ee

and we can also write 
\be
\chi_q(X)=[(H^{X^t} + F^{X^t})(z_q), z_q]
\ee
for the right $t$.

We will now prove that $\chi_q$ is a $C^1$-perturbation of the identity map, which is nothing else but

\be
\begin{array}{ccc}
Id: \mathcal{U}_1 & \rightarrow & \mathcal{U}_1\\
X & \rightarrow & [H^X(z_q), z_q].
\end{array}
\ee

More precisely, we will prove that, independently of $q$,

\be
\|\nabla(\chi_q - Id)\|_{L^\infty} \leq K(\eps),
\ee

for a constant $K(\eps)$ which is an infinitesimal of $\eps$.

We can use the chart $X= [W_1, W_2]=\frac{W_1}{W_2}$ on $\mathcal{U}_1 \subset \CP^1$. Then we must estimate 

$$
\|\nabla(\chi_q - Id)\|_{L^\infty} = \left \| \nabla_X \left( \frac{(H^{X^t} + F^{X^t})(z_q)}{z_q} -\frac{H^X(z_q)}{z_q} \right) \right \|_{L^\infty} 
$$

$$
\leq  \left \| \nabla_z \nabla_X \left( H^{X^t}-H^X + F^{X^t} \right) \right \|_{L^\infty} \leq  
$$

$$
\left \| \nabla_z \nabla_X ( H^{X^t}-H^X ) \right \|_{L^\infty} + \left \| \nabla_z \nabla_X  F^{X^t}  \right \|_{L^\infty}  \leq K(\eps)
$$

thanks to (\ref{eq:estF}).

Thus $\chi_q$ is a diffeomorphism from $\mathcal{U}_1$ to an open subset of $\CP^1$ that tends to $\mathcal{U}$ as $\eps \to 0$. This means that we can invert $\chi_q$ and, for any chosen $q$ we can find $X_q:=(\chi_q)^{-1}([\zeta_q, z_q])$ such that $q \in \Sigma_0^{X_q}$.

\end{proof}

\textbf{Construction of the 3-dimensional surfaces: parallel foliation}. We are always using coordinates as in (\ref{eq:complexcoords}), so that $H_0$ is identified with $\C^2$.

Choose a $J$-invariant plane $X \in \mathcal{U}_1$ passing through $0$. We are going to produce a family of ``parallel'' $3$-dimensional surfaces which foliate a neighbourhood of $0$, where parallel means the following: each 3-surface has tangent planes which are everywhere $\eps$-close to $X \wedge \p_t$ in $C^{2,\nu}$-norm.

This can be done in several ways, we choose the following. Take the vector $v$ in the unit ball inside $X$ which minimizes the distance to $\p_{x_1}$, and set $V=\frac{v}{|v|}$. Parallel transport (in the euclidean sense\footnote{Again, this is just a possible way of doing it: there is no direct geometric meaning.}) the vector $V$ to each point $P$ in the 2-plane $\{z=0, t=0\}$ and consider the family of $J$-invariant planes 
$$\{X_P\}:=\{ V \wedge J_P(V)\}_{P \in \{z=0, t=0\}}.$$

Now, for each $P$, consider the line of points that project to $P$ via $\pi: \R^5 \to \R^4$, and denote them by $(P,t)$. Take the Legendrian, $J$-invariant 2-surface going through the point $(P,t)$ with tangent $X_P^t=  V \wedge J_{(P,t)}(V)$: we will denote it by $\tilde{\mathcal{L}}_{P,t,X}$. Define the 3-dimensional surface 

\be
\label{eq:sigma2}
\Sigma_P^X:= \cup_{t \in (-1, 1)} \tilde{\mathcal{L}}_{P,t,X}.
\ee

As in (\ref{eq:sigma}), this is a smooth 3-surface.

\begin{Prop}
\label{Prop:parallel}
For a fixed $X\in \mathcal{U}_1$, the 3-surfaces $$\{\Sigma_P^X\}_{P \in \{z=0, t=0, |\zeta|\leq 1\}}$$ foliate the set $\{(\zeta, z, t): |\zeta|\leq 1,  |z|\leq 1, |t|\leq \frac{1}{2}\} \subset \C \times \C \times \R=\R^5$.
\end{Prop}

\begin{oss}
Again, in the terminology of \cite{BR}, we are showing that there exist (locally) families of \textit{parallel} foliations made of $3$-surfaces built from embedded, Legendrian, $J$-invariant disks. Each family is determined by a "direction" $X$ at $0$.
\end{oss}

\begin{proof}

Take any $q=(\zeta_q, z_q, t_q) \in B^5 \subset \R^5=\C \times \C \times \R$. Denote by $Q=Q(q,X)$ the intersection point 
\be
Q=Q(q,P):= \Sigma_P^X \cap \{(\zeta, z, t): z=z_q, t=t_q\}.
\ee

This is well-defined because $|T \Sigma_P^X - X| \leq K(\eps)$ and the $3$-plane spanned by $X$ and $\p_t$ is transversal to the $2$-plane $\{(\zeta, z, t): z=z_q, t=t_q\}$ and they have a unique intersection point. By intersection theory, for $\eps$ small enough, $Q$ is uniquely well-defined for all $P \in \{z=0, t=0\}$.

Consider the map

\be
\begin{array}{ccc}
\Gamma_q: D^2_1 \subset \{z=0, t=0\} & \rightarrow & \R^2 \cong \{(\zeta, z, t): z=z_q, t=t_q\} \\
& & \\
P & \rightarrow & Q=Q(q,P)
\end{array}
\ee

With an argument very similar to the one in proposition \ref{Prop:polar}, we can prove that $\Gamma_q$ is a $C^1$-perturbation of the identity map and therefore the family

\be
\{\Sigma_P^X\}_{P \in \{z=0, t=0, |\zeta|\leq 1\}}
\ee

foliates $\{|\zeta|\leq 1,  |z|\leq 1, |t|\leq \frac{1}{2}\}$.

\end{proof}

\medskip

Remark, from the construction of these $3$-surfaces $\Sigma$, that each of them is made by attaching $J$-invariant Legendrian disks along a fiber of the contact structure. This fact yields the following fundamental
  
\medskip 

\textbf{positive intersection property}: each $\Sigma$ constructed above has the property of intersecting positively any transversal $J$-invariant Legendrian.

\medskip

The proof is just analogous to the corresponding corollary 2.1 of \cite{BR}. The key point is that two transversal $J$-invariant 2-planes in a hyperplane $H_p$ intersect themselves positively with respect to the orientation inherited by $H_p$. The 3-surfaces $\Sigma$ are smooth perturbations of a 3-plane of the form $X \wedge \p_t$ for a $J$-invariant 2-plane $X$, so the result follows by continuity.

\medskip

At this stage we have all the ingredients to show theorem \ref{thm:main} by following the proof in sections 3, 4, 5, 6 and 7 of \cite{BR}. For the convenience of the reader, here follows a brief overview of the forementioned proof with references to the corresponding sections of \cite{BR}.

\subsection{Structure of the proof}

A standard blow-up procedure, combined with the almost-monotonicity formula for semi-calibrated cycles\footnote{Recall that in section \ref{j2f} we remarked that $J$-invariant $2$-planes are just the semi-calibrated ones for a suitable $2$-form $\Omega$, therefore an almost-monotonicity formula (see \cite{PR}) holds with respect to the metric induced by $J$ and $\Omega$. Precisely, for any point $x_0$, denoting by $B_r$ the geodesic ball of radius $r$, we have that $\displaystyle \frac{M(C \res B_r(x_0))}{r^2}=R(r) + O(r)$ for a function $R$ which is monotonically non-increasing as $r \downarrow 0$ and tends to the multiplicity at $x_0$ as $r \downarrow 0$, and a function $O(r)$ which is infinitesimal.}, yields that $C$ has a ``stratified'' structure: the multiplicity is well-defined and integer-valued at every point and, for $Q\in \N$, the set $\mathcal{C}^Q$ of points having multiplicity $\leq Q$ is open in $\mathcal{M}$. This allows a localization of the problem by restricting to $\mathcal{C}^Q$ and we can prove the final result by induction on the multiplicity for increasing integers $Q$.

\medskip

A first outcome of the existence of foliations with the positive intersection property, is a self-contained proof of the uniqueness of tangent cones (section 4 of \cite{BR}). This result was proved for general semi-calibrated cycles in \cite{PR} and for area-minimizing ones in \cite{W}, using different techniques.

\medskip

Next, still exploiting the algebraic property of positive intersection, we can locally describe our current $C$ as a multi-valued graph from a two-dimensional disk into $\R^3$ (section 5 of \cite{BR}). The inductive step is divided into two parts: in the first we show that singularities of order $Q$ cannot accumulate onto a singularity of the same multiplicity (sections 5 and 6 of \cite{BR}). In the second part, we prove that singularities of multiplicity $\leq Q-1$ cannot accumulate on a singularity of order $Q$ (section 7 of \cite{BR}).

\medskip

In the first part of the inductive step, we translate the $J$-invariance condition into a system of first-order PDEs for the multi-valued graph (section 5 of \cite{BR}). These equations are ``perturbations'' of the classical Cauchy-Riemann equations, although in this case we have two real variables and \underline{three} functions. We prove a $W^{1,2}$-estimate on the average of the branches of the multi-valued graph (theorem 5.1 of \cite{BR}). Then (section 6 of \cite{BR}) we complete the proof of the first part of the inductive step by suitably adapting the unique continuation argument used in \cite{Ta}.

\medskip

For the second part of the inductive step (section 7 of \cite{BR}) we use a homological argument. On a space modelled on $C \times \R$, we produce a $S^2$-valued function $u$ which allows to ``count'' the lower-multiplicity singularities by looking at its degree on the level sets of $|u|$ (lemma 7.3 of \cite{BR}). A lower bound for the degree (lemma 7.4 of \cite{BR}) then yields the result. This argument is inspired to the one used in \cite{Ta}, however the fifth coordinate induces a more involved and rather lengthy argument.

\section{Final remarks}
\label{ex}
\textbf{Examples}. Let us illustrate some examples where the regularity result of theorem \ref{thm:main} applies.

\medskip

$\bullet$ Let $\mathcal{Y}$ be a Calabi-Yau $3$-fold and denote by $\Theta$ the so-called \textit{holomorphic volume form} and by $\beta$ the symplectic form. Any\footnote{A Calabi-Yau $3$-fold has \textit{real} dimension 6. By hypersurface we mean here that the \textit{real} codimension is 1.} hypersurface $M^5 \subset \mathcal{Y}$ of contact type inherits a contact structure from the symplectic structure of $\mathcal{Y}$ (see \cite{MS}), namely the structure associated to the one-form $\alpha=\iota_N \beta$, where $N$ denotes a Liouville vector field and $\iota$ denotes the interior product. If $\mathcal{Y}=\C^3$ (with the standard complex structure) then $M^5$ can be, for example, the boundary of any smooth, star-shaped (with respect to the origin) domain and $N$ the radial vector field. The $2$-form 
$$\om= \iota_N \Theta$$
(restricted to $M$) is a horizontal two-form for this contact structure and satisfies $\om \wedge d\alpha=0$, $\om \wedge \om=(d\alpha)^2$. Moreover $\om$ is of comass $1$, it is therefore a semi-calibration. Then we deduce from proposition \ref{Prop:calibr} that integral cycles (semi)calibrated by $\om$ are smooth except possibly at isolated point singularities.

\medskip

$\bullet$ In the previous framework, we can also recover the Special Legendrians in $S^5$. Consider the canonical embedding $\mathcal{E}: S^5 \hookrightarrow \mathbb{C}^3$ and denote by $N$ the radial vector field $N:=r \frac{\partial}{\partial r}$ in $\mathbb{C}^3$. The sphere inherits from the symplectic manifold $\displaystyle (\mathbb{C}^3, \sum_{i=1}^3 dz^i \wedge d\overline{z}^i)$ the contact structure given by the form 
$$\gamma := \mathcal{E}^\ast \iota_N (\sum_{i=1}^3 dz^i \wedge d\overline{z}^i).$$
The $3$-form $\Omega = Re(dz^1 \wedge dz^2 \wedge dz^3)$ is known as \textit{Special Lagrangian calibration} in $\C^3$. The \textit{Special Legendrian semi-calibration} is defined as the following $2$-form on $S^5$ (of comass 1):
$$\omega := \mathcal{E}^\ast \iota_N \Omega = Re(z_1 dz^2 \wedge dz^3 + z_2 dz^3 \wedge dz^1 + z_3 dz^1 \wedge dz^2).$$
$\om$-semicalibrated cycles are known as Special Legendrians.

We remark that there is a natural projection $\Pi:S^5 \to \CP^2$ (Hopf projection) whose kernel is given by the Reeb vectors of the contact distribution. The Reeb vector field can be integrated to obtain closed orbits which are nothing but the Hopf fibers $e^{i \theta} p$, for $p \in S^5$ and $\theta \in [0, 2\pi)$. Every Special Legendrian is projected via $\Pi$ to a minimal Lagrangian in $\CP^2$ (see \cite{R}).

\medskip

$\bullet$ The same as in the first example of this section applies, more generally, in a contact 5-manifold with an SU(2)-structure, as defined in \cite{CS}. In the mentioned work, it is proved that, if the data are analytic, then this 5-manifold embeds in a Calabi-Yau $3$-fold. Our regularity result, however, only requires the SU(2)-structure on a contact 5-manifold.

\medskip

$\bullet$ Let us look at the following situation, \cite{Ri}. Let $S^3$ be the unit sphere in $\R^4$ and consider the following Riemannian $5$-manifold
$$N^5 = \{(e_1, e_2) \in S^3 \times S^3: \langle e_1, e_2 \rangle_{\R^4}=0\}, $$ 
endowed with the metric inherited from $\R^4 \times \R^4$.

The tangent space to $N^5$ at a point $(e_1, e_2)$, is identified with those $U =(U_1, U_2) \in \R^4 \times \R^4$, such that $\langle U_1, e_1 \rangle_{\R^4}=0$, $\langle U_2, e_2 \rangle_{\R^4}=0$ and $\langle U_1, e_2 \rangle_{\R^4}+  \langle U_2, e_1 \rangle_{\R^4}=0$.

At every $(e_1,e_2) \in N^5$, consider the tangent vector $v=(-e_2, e_1) \in T_{(e_1, e_2)}N^5$ and take the orthogonal hyperplane $H_{(e_1, e_2)}= v^\perp \subset T_{(e_1, e_2)} N^5$. The distribution $H$ defines a contact structure on $N^5$. It can be described by the one-form $\alpha_{(e_1, e_2)}(U)=\frac{1}{2}\left(\langle e_1, U_2 \rangle_{\R^4}-\langle e_2, U_1 \rangle_{\R^4}\right)$ with associated symplectic form $\Om(U,V)=\langle U_1, V_2 \rangle_{\R^4}-\langle V_1, U_2 \rangle_{\R^4}$.

By integrating the Reeb vectors $v$, we get closed fibers isomorphic to $S^1$ of the form $$\{(\cos \theta e_1 - \sin \theta e_2, \sin \theta e_1 + \cos \theta e_2)\}_{\theta \in [0,2\pi)}.$$
The map\footnote{Here $G_2(\R^4)$ denotes the Grassmannian of $2$-planes in $\R^4$. We have the identification $G_2(\R^4) \cong \CP^1 \times \CP^1$ by splitting into the self-dual and anti self-dual components.}
$$\begin{array}{ccc}
\Pi: N^5 & \to & G_2(\R^4) \cong \CP^1 \times \CP^1 \\
(e_1, e_2) & \to & e_1 \wedge e_2
\end{array}$$
is an orthogonal projection whose kernel is given by the Reeb vectors. 

Define the following $2$-form on $N^5$
$$\om(U,V):= e_1 \wedge e_2 \wedge (U_1 \wedge V_2 - V_1 \wedge U_2), $$ 
$$\text{for } U,V \text{ tangent vectors to } N^5 \text{ at }(e_1, e_2).$$
It can be checked that $\om$ is a horizontal form of comass $1$ and our regularity result applies to $\om$-semicalibrated cycles.

\medskip

$\bullet$ In \cite{TV}, the authors introduce the notions of Contact Calabi-Yau manifolds and Special Legendrians in Contact Calabi-Yau manifolds. With regard to the notation in \cite{TV}, the two-form $Re \;\epsilon$ is a calibration (it is assumed to be closed) and Proposition \ref{Prop:calibr} yields the regularity of cali\-brated cycles in dimension $5$. We still get the regularity result if we drop the closedness assumption on $\epsilon$.

\medskip

\textbf{What else?} We conclude with a short motivational digression regarding theorem \ref{thm:main}, in connection to general calibrations.

\medskip

For a general calibrating $2$-form $\varphi$ in a $5$-dimensional manifold $M$, let us look, at every point, at the set $\mathcal{G}_\varphi$ of calibrated $2$-planes: as explained in \cite{HL} (Thm. II 7.16) or \cite{J} (Thm. 4.3.2), there exist suitable orthogonal coordinates at the chosen point such that $\mathcal{G}_\varphi$ is the same as the set of $2$-planes calibrated by one of the following canonical forms
$$dx^1 \wedge dx^2 + dx^3 \wedge dx^4 \; \text{ or }\;  dx^1 \wedge dx^2.$$
At the points where the first case is realized, we can define an almost complex structure $J$ such that calibrated $2$-planes are identified with the $J$-invariant ones. If moreover the manifold $M$ is contact, then, as we already discussed, a calibrated manifold (or also an integer multiplicity rectifiable current) can have as tangents only those $J$-invariant planes which are Lagrangian for the symplectic form on the horizontal distribution. Therefore, if we require the calibration to admit, for every point $p$ and calibrated $2$-plane $\Pi$ at $p$, a calibrated submanifold passing through $p$ with tangent $\Pi$, the corresponding $J$ must fulfil conditions (\ref{eq:Jlagr}) and (\ref{eq:anticomp}).

In many instances, a calibration is considered interesting if it admits a lot of calibrated submanifolds\footnote{This point of view is present both in \cite{HL} and in \cite{J}.}. Indeed, the richer the family of calibrated submanifolds is, more examples of area-minimizing surfaces and their possible singularities can we get. On a contact $5$-manifold, therefore, our assumption on $J$ includes, in some sense, the most generic cases of calibrations. 

This $5$-dimensional situation can be considered as the analogue of the one addressed in \cite{RT1} and \cite{Ta} in dimension $4$, or in \cite{RT2} for general even dimension, where the corresponding regularity for $J$-holomorphic cycles is proven.


\begin{thebibliography}{99}


\bibitem{BR} Bellettini, Costante and Rivi\`ere, Tristan {\it The regularity of Special Legendrian integral cycles}, preprint 2009.
\bibitem{B} Blair, David E. {\it Riemannian geometry of contact and symplectic manifolds}, {Progress in Mathematics}, {203}, {Birkh\"auser Boston Inc.}, {Boston, MA}, {2002}, {xii+260}.
\bibitem{CS} Conti, Diego and Salamon, Simon {\it Generalized {K}illing spinors in dimension 5}, {Trans. Amer. Math. Soc.}, {359}, {11}, {5319--5343}, {2007}.
\bibitem{DT} Donaldson, Simon K. and Thomas, Richard P. {\it Gauge Theory in higher dimensions} in
"The
geometric Universe" (Oxford, 1996), Oxford Univ. Press, 1998, 31-47.
\bibitem{F} Federer, Herbert {\it Geometric measure theory.}, Die Grundlehren der mathematischen Wissenschaften, Band 153, Springer-Verlag New York Inc., New York, 1969, xiv+676.
\bibitem{G} Giaquinta, Mariano and Modica, Giuseppe and Sou{\v{c}}ek,
              Ji{\v{r}}{\'{\i}}  {\it Cartesian currents in the calculus of variations. {I}.}, Ergebnisse der Mathematik und ihrer Grenzgebiete. 3. Folge. A
              Series of Modern Surveys in Mathematics [Results in
              Mathematics and Related Areas. 3rd Series. A Series of Modern
              Surveys in Mathematics], 37, Cartesian currents, Springer-Verlag, Berlin, 1998, xxiv+711.
\bibitem{HL} Harvey, Reese and Lawson, H. Blaine Jr. {\it Calibrated geometries}, {Acta Math.},{148}, {47--157},{1982}.
\bibitem{J} Joyce, Dominic {\it Riemannian holonomy groups and calibrated geometry.}, {Oxford Graduate Texts in Mathematics},{12}, {Oxford University Press}, {Oxford}, {2007}, {x+303}.
\bibitem{MS} McDuff, Dusa and Salamon Dietmar {\it Introduction to symplectic topology}, {Oxford Mathematical Monographs}, {2}, {The Clarendon Press Oxford University Press}, {New York}, {1998}, {x+486}.
\bibitem{M} Morgan, Frank {\it Geometric measure theory: a beginner's guide}, {Fourth edition}, {Elsevier/Academic Press, Amsterdam}, {2009}, {viii+249}.
\bibitem{PR} Pumberger, David and Rivi{\`e}re, Tristan  {\it Uniqueness of tangent cones for semi-calibrated 2-cycles.}, to appear in Duke Math. J.
\bibitem{R} Reckziegel, Helmut {\it Horizontal lifts of isometric immersions into the bundle space of a pseudo-Riemannian submersion}, Global differential geometry and global analysis 1984 (Berlin, 1984),  264--279,
Lecture Notes in Math., 1156, Springer, Berlin, 1985. 
\bibitem{Ri} Rivi\`ere, Tristan, \textit{private communication}.
\bibitem{RT1} 
Rivi\`ere, Tristan; Tian, Gang {\it The singular set of $J$-holomorphic maps into projective algebraic varieties.} 
 J. Reine Angew. Math.  570  (2004), 47--87. 58J45
\bibitem{RT2} Rivi{\`e}re, Tristan and Tian, Gang  {\it The singular set of 1-1 integral currents.}, Ann. of Math. (2), Annals of Mathematics. Second Series, 169, 2009, 3, 741-794.
\bibitem{SYZ} Strominger, Andrew and Yau, Shing-Tung and Zaslow, Eric \textit{Mirror symmetry is T-duality}, Nuclear
Phys. B 479 (1996), 243-259.
\bibitem{Ta} Taubes, Clifford Henry {\it Seiberg Witten and Gromov invariants for symplectic 4-manifolds}, "$\rm Gr\Rightarrow SW$: from pseudo-holomorphic curves to Seiberg-Witten solutions ".   163--273, First Int. Press Lect. Ser., 2, Int. Press, Somerville, MA, 2000.
\bibitem{T} Tian, Gang {\it Gauge theory and calibrated geometry I}, {Ann. of Math. (2)}, Annals of Mathematics. Second series. {151}, {1}, {193-268}, {2000}.
\bibitem{TV} Tomassini, Adriano and Vezzoni, Luigi {\it Contact Calabi-Yau manifolds and Special Legendrian submanifolds}, {Osaka J. Math.}, {45}, {127--147}, {2008}.
\bibitem{W} White, Brian {\it Tangent cones to two-dimensional area-minimizing integral
              currents are unique.} , Duke Math. J., Duke Mathematical Journal, 50, 1983, 1, 143--160.

\end{thebibliography}
\end{document}